\documentclass[12pt]{amsart}
\usepackage{amsmath}
\usepackage{amsthm}
\usepackage{amssymb}
\usepackage{mathtools}
\usepackage{bbm}
\usepackage{enumerate}

\theoremstyle{plain}
\newtheorem{theorem}{Theorem}[section]

\newtheorem{lemma}[theorem]{Lemma}
\newtheorem{proposition}[theorem]{Proposition}

\theoremstyle{definition}

\theoremstyle{remark}
\newtheorem{remark}[theorem]{Remark}

\DeclareFontFamily{U}{mathx}{}
\DeclareFontShape{U}{mathx}{m}{n}{<-> mathx10}{}
\DeclareSymbolFont{mathx}{U}{mathx}{m}{n}
\DeclareMathAccent{\widehat}{0}{mathx}{"70}
\DeclareMathAccent{\widecheck}{0}{mathx}{"71}

\begin{document}
	\title[Solutions to Fifth-Order KP II Scatter]{Solutions to the Fifth-Order KP II Equation Scatter}
	\author{Peter A. Perry, Camille Schuetz}
	\address{ Department of Mathematics, University of Kentucky, Lexington, Kentucky 40506--0027, U.S.A.}
	\email{pperr0@uky.edu}
	\address{ Department of Mathematics, University of Wisconsin, Platteville, 
		Platteville, Wisconsin, 53818, U.S.A.}
		
	\email{feltonc@uwplatt.edu}

\begin{abstract}
	The fifth-order KP II equation describes dispersive long waves in two space dimensions. In this paper we show that solutions with small initial data scatter to solutions of the associated linear fifth-order equation. In particular, we establish the existence of nonlinear wave operators mapping the initial data to scattering asymptotes, and show that the nonlinear wave operators have inverses in a neighborhood of the origin. Our paper uses techniques developed for the third-order KP II equation by Hadac, Herr, and Koch. 
\end{abstract}

\medskip
	
\maketitle

	
\newcommand{\loc}{\mathrm{loc}}
\newcommand{\rc}{\mathrm{rc}}

\newcommand{\diff}{\partial}
\newcommand{\eps}{\varepsilon}
\newcommand{\dotarg}{\, \cdot \,}

\newcommand{\C}{\mathbb{C}}
\newcommand{\N}{\mathbb{N}}
\newcommand{\R}{\mathbb{R}}
\newcommand{\Z}{\mathbb{Z}}

\newcommand{\bfP}{\mathbf{P}}

\newcommand{\frakt}{\mathfrak{t}}

\newcommand{\calF}{\mathcal{F}}
\newcommand{\calS}{\mathcal{S}}
\newcommand{\calZ}{\mathcal{Z}}

\newcommand{\dotH}{\dot{H}}

\newcommand{\norm}[2][ ]{\left\Vert #2 \right\Vert_{#1}}
\newcommand{\bigO}[2][ ]{\mathcal{O}_{#1} \left( #2 \right)}


%
%
%
%

%
%




\newcommand{\sgn}{\mathrm{sgn}}
\newcommand{\ran}{\mathrm{Ran}}
\newcommand{\esssup}{\mathrm{ess} \, \mathrm{sup}}
\newcommand{\supp}{\mathrm{supp}}


\section{Introduction}
	The fifth-order Kadomtsev-Petviashvilli II equation is the nonlinear dispersive equation
	\begin{equation}
	\left\{
	\label{KPII-5}
		\begin{aligned}
			\diff_t u + \alpha \diff_x^3 u + \beta \diff_x^5 u + u \diff_x u + \diff_x^{-1} \diff_y^2 u &= 0,\\
			u(0,x,y) &= u_0(x,y)	.
		\end{aligned}
	\right. 
	\end{equation}
	It is an analogue in two space dimensions of the Kawahara equation \cite{Kawahara72}
$$
\diff_t w + w\diff_x w + \alpha \diff^3_x w+ \diff_x^5 w = 0,
$$
much as the third-order KP II equation (i.e., \eqref{KPII-5} with $\beta=0$ and $\alpha=1$) is a two-dimensional analogue of the KdV equation
$$
\diff_t w +  w \diff_x w + \diff^3_x w = 0.
$$
Both equations occur in the modeling of dispersive long waves in two space dimensions, where the $\diff_y^2 u$ term models weak dispersive effects in the $y$-direction.
In \eqref{KPII-5},  $\beta<0$ and $\alpha >0$;  henceforth we take $\beta=-1$. For further discussion of the fifth-order KP II equation,  see Klein and Saut's monograph \cite[section 5.3.12]{KS2021} and the introductions to \cite{ST99,ST00}.

In this paper we will prove that solutions of \eqref{KPII-5} with small initial data scatter as $t \to \pm \infty$ to solutions of the linear problem 
\begin{equation}
\label{KPII-5-linear}
\left\{
\begin{aligned}
v_t + \alpha \diff_x^3 v + \beta\diff_x^5 v + \diff_x^{-1} \diff_y^2 v
&=0,\\
v(0,x,y)&=u_\pm(x,y).
\end{aligned}
\right.
\end{equation}
with initial data $u_\pm$.  We will also prove continuity and analyticity of the wave operators
$$W_{\pm}: u_0 \to u_{\pm}$$ and their inverses
$$V_{\pm}: u_\pm \to u_0$$
in suitable function spaces. A detailed exposition of our results may be found in the doctoral thesis of the second author \cite{CS2023}.

Our work is inspired by the work of Hadac, Herr, and Koch \cite{HHK09,HHK10} on the third-order KP II equation 
\begin{equation}
	\label{KPII-3}
		\diff_t u + \diff_x^3 u + u \diff_x u + \diff_x^{-1} \diff_y^2 u = 0 
\end{equation}
and, as we will explain, we use a number of ideas and techniques from their paper in our analysis. Both the third- and fifth-order KP II equation are nonlinear dispersive equations of the form
	$$ \diff_t u + ih(D)u = NL(u) $$
	where $h$ is the dispersion relation, i.e.,
    \begin{equation}
        \label{KPII-5-dispersion}
        h(\xi,\eta) = \alpha \xi^3 + \beta \xi^5 -\xi^{-1} \eta^2 
    \end{equation}
	(with $\alpha=1$, $\beta=0$ for \eqref{KPII-3}), and the nonlinearity $NL(u)$ is quadratic. 
    For both equations, the symbol $h$ satisfies a non-resonance condition: letting $\zeta=(\xi,\eta)$, the equation 
    \begin{equation}
        \label{non-resonance-condition}
            h(\zeta-\zeta')+ h(\zeta') = h(\zeta)
    \end{equation}
	has no nonzero solutions. For the third-order KP II equation, and for the nonlinear Klein-Gordon and quadratic Schr\"{o}dinger equations considered in the thesis of Schottdorf \cite{Schottdorf13}, the non-resonance condition leads to key estimates on the nonlinearity which lead to global existence results for small data. In the thesis, Schottdorf remarks that, based on his treatment of non-resonant and resonant dispersive equations, one should expect to obtain global existence and scattering for second-order dispersive PDE when the non-resonance condition holds (see the discussion in \S 7.4 of Schottdorf's
    thesis; see also Germain's expository paper \cite{Germain2010} on space-time resonances, which discusses resonances in a broader context and motivates the condition \eqref{non-resonance-condition}). It is therefore natural to expect that the methods used by Hadac, Herr, and Koch in \cite{HHK09} and by Schottdorf in his thesis should also be effective for the fifth-order KP II equation. It is also natural to expect that the spaces $U^p$ and $V^p$ used in their work should provide a useful framework for the analysis of \eqref{KPII-5}; we define and discuss these spaces in what follows.

Hadac, Herr, and Koch obtain global existence and scattering for the third-order KP II equation for small initial data in the homogeneous Sobolev space $\dot{H}^{\frac12,0}(\R^2)$, the space of tempered distributions with
$$ \norm[\dot{H}^{-\frac12,0}(\R^2)]{u}=
	\left( \int_{\R^2} |\xi|^{-1} |\widehat{u}(\xi,\eta)|^2 \, d\xi \, d\eta\right)^\frac12
$$ 
finite. This space is the borderline of the sub-critical regime for the third-order KP II equation: the solution space of \eqref{KPII-3} is invariant under the scaling
	$$ u(t,x,y) \to \lambda^2 u(\lambda^3 t, \lambda x, \lambda^2 y) $$
	and the space $\dot{H}^{-\frac12,0}(\R^2)$ is invariant under this symmetry.

 We will also prove global existence and scattering for equation \eqref{KPII-5} for small initial data in $\dot{H}^{-\frac12,0}(\R^2)$: although this space is not critical for the fifth-order KP II equation,\footnote{For $\alpha=0$, equation \eqref{KPII-5} is critical for $H^{s_1,s_2}(\R^2)$ where $s_1+3s_2=-2$; see the discussion in the introduction to Li and Shi \cite{LS13}.} it is effectively scale-invariant for low frequencies, a fact which plays an important role in the estimates.

We will regard \eqref{KPII-5} as equivalent to the integral equation
\begin{equation}
\label{KPII-5-int}
u(t) 	=	S(t)u_0  - I(u,u)
\end{equation}
in a suitable function space,  where $S(t)$ is the solution operator
for the linear problem \eqref{KPII-5-linear} and 
\begin{equation}
\label{KPII-5-IT}
I(u,v)(t) = \frac12 \int_0^t S(t-s) (uv)_x (s) \,  ds
\end{equation}
is a bilinear map. As in \cite{HHK09} we look for solutions in $C(\R,\dot{H}^{-\frac12,0}(\R^2))$ 

We will study the integral equation \eqref{KPII-5-IT} in certain function spaces $\dot{Z}^{-\frac12}$ and $\dot{Y}^{-\frac12}$ contained in $C(\R,\dot{H}^{-\frac12,0})$ and analogous to those used in \cite{HHK09} (see section \ref{subsec:YZ} for a discussion of these spaces). As we will explain, finding a solution in $\dot{Z}^{-\frac12}$ immediately guarantees the existence of scattering asymptotes. The spaces $\dot{Z}^{-\frac12}$ and $\dot{Y}^{-\frac12}$ are built from spaces $U^p$ and $V^p$ of functions $u:\R \to L^2(\R)$ with finite $p$-variation: the spaces $V^p$ were introduced for functions on the real line by Wiener in \cite{Wiener24} 


while the dispersive versions of the spaces $U^p$ and $V^p$  were introduced by Tataru in \cite{T-unpublished}.
These spaces were further developed by 
by Koch and Tataru in \cite{KT05,KT07,KT2018}. 
Functions in a related space $V^p_0$ have asymptotic limits at infinity; see Proposition \ref{prop:V0} and the accompanying discussion. As in \cite{HHK09}, this fact will imply that $S(-t)u(t)$ has an asymptotic limit as $t \to \pm \infty$ when $u \in \dot{Y}^{-\frac12}$ (see Lemma \ref{lemma.Y.scat}), which establishes the existence of scattering asymptotes.

For $\delta > 0$, let
\begin{align}
	\label{dotB}
	\dot{B}_\delta&=	\left\{ 
					u_0 \in \dot{H}^{-\frac12,0}(\R^2): 
					\norm[\dot{H}^{-\frac12,0}(\R^2)]{u_0} < \delta
				\right\}
\end{align}
It is important to note that $\dot{Z}^{-\frac12} \subset C(\R,\dot{H}^{-\frac12,0}(\R^2))$. 
We will prove:

\begin{theorem}
	\label{thm:main} 
	There exist $\delta>0$ so that for any 
	$u_0 \in \dot{B}_\delta$, the KPII equation \eqref{KPII-5} has a unique solution $u \in \dot{Z}^{-\frac12}$ depending continuously on the initial data $u_0$. Moreover, the limits 
		$$ u_+ = \lim_{t \to \infty} S(-t) u(t), \quad
		   u_- = \lim_{t \to -\infty} S(-t) u(t)
		$$
	exist. If $u_0 \in \dot{B}_\delta \cap L^2$, then $u(t) \in L^2(\R^2)$ for all $t$ and the $L^2$ norm is conserved. The maps
	$$ W_\pm : \dot{B}_\delta \ni u_0 \to u_\pm \in \dot{H}^{-\frac12,0}(\R^2)$$
	are analytic with $\norm[L^2]{u_\pm} = \norm[L^2]{u_0}$ if $u_0 \in \dot{B}_\delta \cap L^2(\R^2)$. 
    Finally, there is a ball $B_\eta$ in $\dot{H}^{-\frac12,0}(\R^2)$ so that  
	$$ V_\pm: B_\eta  \ni u_\pm \to u_0 \in \dot{B}_\delta $$
	are analytic with $\norm[L^2]{u_0} = \norm[L^2]{u_\pm}$ if $u_\pm \in \dot{H}^{-\frac12,0}(\R^2) \cap L^2(\R^2)$.
	
\end{theorem}

To our knowledge, Theorem \ref{thm:main} is the first proof of scattering for \eqref{KPII-5}. On the other hand, the well-posedness result obtained here is far from optimal.

Indeed, there is a substantial literature on well-posedness for \eqref{KPII-5}.    Saut and Tzvetkov \cite{ST99,ST00} proved global well-posedness for 
\eqref{KPII-5} having initial data $u_0 \in L^2(\R^2)$ with $|\xi|^{-1} \widehat{u_0}(\xi,\eta) \in \calS'(\R^2)$.  Hadac \cite{Hadac08} studied the dispersion-generalized KP II equation
$$ \left(u_t - |D_x|^\gamma u_x + (|u|^2)_x\right) + u_{yy} = 0$$
for $\frac43 \leq \gamma \leq  6$  (if $\gamma=4$,  his equation
coincides with \eqref{KPII-5} for $\alpha=0$, $\beta<0$) and proved
global well posedness for initial data in the space
$$H^{s_1,s_2}(\R^2) = 
	\left\{ 
				u_0 \in \calS'(\R^2): 
					\int 
						\langle \xi \rangle^{2s_1} 
						\langle \eta \rangle^{2s_2} 
						|\widehat{u_0}(\xi,\eta)|^2  
					\, d\xi \,  d\eta 
				< \infty 
	\right\}.
$$
if $s_1 > -\frac12$ and $s_2 \geq 0$.
Isaza,  L\'{o}pez, and Mej\'{i}a \cite{ILM06} proved that solutions of the 
fifth-order KP II equation \eqref{KPII-5} (with $\alpha=0$) are locally well-posed in anisotropic Sobolev spaces $H^{s_1,s_2}(\R^2)$ with $s_1 > -\frac{5}{4}$ and $s_2 \geq 0$,  and globally well-posed in $H^{s,0}$ if $s>-\frac47$. 
More recently,  Li and Shi \cite{LS13} proved local well-posedness
for \eqref{KPII-5} with $\alpha=0$, $\beta=-1$ and initial data in the
space $H^{s_1,s_2}(\R^2)$ with $s_1 \geq -\frac54$ and $s_2 \geq 0$.  
For this reason, one might expect scattering results for \eqref{KPII-5} for initial data in rougher spaces than $\dot{H}^{-\frac12,0}(\R^2)$, but we have not considered this question here.

ut we have not considered this question here.

\section{Function Spaces and Fourier Decompositions}

To analyze the integral equation \eqref{KPII-5-int}, we introduce spaces $\dot{Y}^s$ and $\dot{Z}^s$ similar to those used in \cite{HHK09}. To define them, we first recall the spaces $V^p$ and $U^p$ of certain functions $u: \R \to L^2(\R^2)$ with 
	$$\norm[L^\infty]{u} \coloneqq \sup_t \norm[L^2]{u(t)}$$ 
	finite. The $V^p$ spaces were introduced by Wiener \cite{Wiener24} for functions on the line;
    
    the dispersive versions of $U^p$ and $V^p$ were first introduced by Tataru in \cite{T-unpublished}
    (see \cite[Section 2]{HHK09}, \cite{Koch14}, \cite{KT05}, Koch's lectures in \cite{KTV14}, Appendix B of \cite{KT2018}, and the paper of Candy-Herr \cite{CH18} for details and further references). 

We then define the $\dot{Z}^s$ and $\dot{Y}^s$ spaces. 

\subsection{$U^p$ and $V^p$ Spaces}
\label{subsec:UV}
	
	Here we follow the conventions of \cite{HHK09} including the modifications discussed in \cite{HHK10}. Let $\bfP$ denote the collection of partitions 
$$
	\tau = \{ t_i \}_{i=1}^N:  N \in \N, \quad -\infty =t_0 < t_1 < \ldots < t_N =  \infty
$$
and let $\bfP_\infty$ denote the collection of partitions 
$$  \{ t_i \}_{i=1}^N: \quad N \in \N, \quad-\infty < t_0 < \ldots < t_{N-1} < t_N \leq \infty$$
For a fixed $p$ with $1 < p < \infty$, the space $V^p$ consists of functions $u:\R \to L^2(\R^2)$  such that the limits $\lim_{t \to \pm \infty} u(t)$ exist and so that the norm
$$ \norm[V^p]{u} = \sup_{\tau \in \bfP} \left( \sum_{i=1}^{N} \norm[L^2]{v(t_{i})-v(t_{i-1})}^p\right)^{\frac1p}$$
is finite, where we make the notational conventions that $v(\infty)=0$ and $v(-\infty)=\lim_{t \to -\infty}v(t)$. Note that the $V^p$ norm bounds $\norm[L^2]{v(-\infty)}$ by considering the two-point partition $\{ -\infty,\infty\}$ and $\sup_{t \in \R} \norm[L^2]{u(t)}$ by considering the three-point partition $\{-\infty,t_1,\infty\}$. The space $V^p$ is a Banach space.  Functions in $V^p$ have right-hand limits in $[-\infty,\infty)$ and left-hand limits in $(-\infty,\infty]$. 

We denote by $V^p_-$ the space of all functions $v: \R \to L^2$ with $v(-\infty)=0$, $\lim_{t \to \infty} v(t)$ exists, and $\norm[V^p]{v} < \infty$, and by $V^p_{{\rc}}$ (resp.\ $V^p_{-,{\rc}}$) the closed subspace of $V^p$ consisting of right-continuous functions (resp.\ the closed subspace of $V^p_-$ consisting of right-continuous functions).

The following proposition (see \cite[Proposition 2.4(i)]{HHK09}) will be very important in our analysis of scattering.

\begin{proposition} \cite{HHK09}
	\label{prop:V0}
	Fix $p$ with $1 < p < \infty$.
	Suppose that $v: \R \to L^2(\R^2)$ be such that 
	$$ \norm[V^p_0]{v} \coloneqq \sup_{\tau \in \mathbf{P_\infty}} \left( \sum_{i=1}^{N} \norm[L^2]{v(t_{i}) - v(t_{i-1})}^p  \right)^\frac1p
	$$
	is finite. Then $\lim_{t \downarrow t_0} v(t)$ exists for all $t_0 \in [-\infty,\infty)$ and $\lim_{t \uparrow t_0} v(t)$ exists for all $t_0 \in (-\infty,\infty]$. Moreover,
	$\norm[V^p]{v} = \norm[V^p_0]{v}$.
	\end{proposition}

Next, we define the $U^p$ spaces. For a partition $\tau \in \bfP$, denote by $\chi_{[t_{i-1},t_{i})}$ the characteristic function of $[t_{i-1},t_i)$. The space $U^p$ is an atomic space with atoms
$$ a(t) = \sum_{i=1}^{N} \chi_{[t_{i-1},t_i)}(t) \psi_{i-1}, \quad \sum_{i=0}^{N-1} \norm[L^2]{\psi_i}^p \leq 1, \quad \psi_0 \equiv 0$$
and elements of the form 
$$ u = \sum_{i=1}^\infty \lambda_i a_i $$
where each $a_i$ is a $U^p$ atom and $\sum_{i=1}^\infty |\lambda_i|< \infty$. The norm on $U^p$ is
$$ 
	\norm[U^p]{u} = \inf 
						\left\{ 
							\sum |\lambda_i|: 
							u = \sum \lambda_i a_i 
						\right\}.
$$ 
The space $U^p $ is continuously embedded in $ V^p_{-,{\rc}}$ (see \cite[Proposition 2.4(iii)]{HHK09}).

To estimate $U^p$ norms we will use the following duality principle (see Hadac-Herr-Koch \cite{HHK09}, Theorem 2.8, Proposition 2.10, and Remark 2.11, and, for the form used here, Candy-Herr \cite{CH18}, Theorem 5.1).
\begin{equation}
	\label{Up.norm.dual}
	\norm[U^p]{u} = 
		\sup_{
			\substack{
						\diff_t \varphi \in C_0^\infty(\R^3)\\[2pt] 
						\norm[V^p]{\varphi} \leq 1
					}
			} 
			\left| 
				\int \langle \diff_t  \varphi ,u \rangle_{L^2} \, dt 
			\right|
\end{equation}

Next, we define spaces $U^p_S$ and $V^p_S$ adapted to the linear evolution $S(t)$ for \eqref{KPII-5-linear}. The space $U^p_S$ (resp.\ $V^p_S$) is the space of functions $u$ with $S(-t)u(t) \in U^p$ (resp.\ $S(-t)u(t) \in V^p$) with norms
$$ 
	\norm[U^p_S]{u} = \norm[U^p]{S(-t) u(t)}, \quad \norm[V^p_S]{u} = \norm[V^p]{S(-t) u(t)}.
$$ 
We make an analogous definition for $V^p_{-,{\rc},S}$. 
For exact solutions of the linear problem \eqref{KPII-5-linear}, the norms $\norm[U^2_S]{\cdot}$ and $\norm[V^2_S]{\cdot}$ recover the $L^2$ norm of the initial data. Note that, if $u \in V^p_S$, then the limits $\lim_{t \to \pm \infty} S(-t)u(t)$ exist in $L^2(\R^2)$ since $V^p$ functions have limits in $L^2$ as $t \to \pm \infty$.  We also have the following analogue of \eqref{Up.norm.dual}:
\begin{equation}
	\label{Ups.norm.dual}
	\norm[U^p_S]{u} = 
		\sup_{
			\substack{
						\diff_t \varphi \in C_0^\infty(\R^3)\\[2pt] 
						\norm[V^p]{\varphi} \leq 1
					}
			} 
			\left| 
				\int \langle \diff_t  \varphi ,S(-t)u \rangle_{L^2} \, dt 
			\right|.
\end{equation}

\subsection{{The Transfer Principle for $U^p_S$}}

\label{subsec:transfer}

The $U^p_S$ spaces obey a ``transfer principle''
which comes from the the fact that $U^p_S$ atoms are superpositions of 
the form $$\sum_{i=1}^N \chi_{[t_{k-1},t_k)}(t) S(t) \psi_{k-1}$$ where
$\sum_{k=1}^N \chi_{[t_{k-1},t_k)}(t) \psi_{k-1}$ is a $U^p$ atom.

\begin{proposition}[{\cite[Proposition 2.19(i)]{HHK09}}]
\label{prop:transfer}
Suppose that $T_0$ is an $n$-linear mapping 
$$ T_0: L^2(\R) \times \ldots \times L^2(\R^2) \to L^1_{\mathrm{loc}}(\R^2;\C). $$
Fix $p$ and $q$ with $1 \leq p,q \leq \infty$ and suppose that
			$$ 
			\norm[L^p_t(\R,L^q(\R^2))]{T_0(S(\cdot) \varphi_1, \ldots, S(\cdot) \varphi_n)} 
			\lesssim \norm[L^2(\R^2)]{\varphi_1} \ldots \norm[L^2(\R^2)]{\varphi_n}.
			$$
			There exists a map 
			$T:U^p_S \times \ldots \times U^p_S \to L^p(\R_t,L^q(\R^2))$
			extending $T_0$ with
			\begin{equation}
			\label{transfer.norm.1}
			\norm[L^p_t(\R,L^q(\R))]{T(u_1, \ldots, u_n)}
				\lesssim	\norm[U^p_S]{u_1} \times \ldots \times \norm[U^p_S]{u_n}.
			\end{equation}
			such that
			$$T(u_1,\ldots, u_n)(x,y) = T_0(u_1(t),\ldots u_n(t))(x,y) $$
			a.e.
\end{proposition}

A simple application of the transfer principle and the Strichartz estimate proven in Lemma \ref{lemma:Strichartz} of what follows gives the following estimates (compare \cite{HHK09}, Corollary 2.21).
\begin{align}
	\label{L4.Vp}
	\norm[L^4(\R^3)]{u} & \lesssim \norm[V^p_{-,S}]{u}, \quad 1 \leq p < 4.
\end{align}

\subsection{The Spaces $\dot{Y}^s$ and $\dot{Z}^s$}
\label{subsec:YZ}

For $s < 0$, the spaces $\dot{Y}^s$ and $\dot{Z}^s$ are modifications of $V^2_S$ and $U^2_S$ adapted to functions with low regularity.
To define these spaces we first recall the Littlewood-Paley projections (see for example \cite[Appendix A]{Tao06} for a discussion).  

Let $\varphi \in C_0^\infty(\R)$ with $\varphi(\xi)=1$ for $|\xi| \leq 1$ and $\varphi(\xi)=0$ for $|\xi| \geq 2$. Let $\psi(\xi) = \varphi(\xi)-\varphi(2\xi)$. A dyadic integer $N$ is an integer $N=2^k$ for some $k \in \Z$. We set $\psi_N(\xi) = \psi(\xi/N)$ and we denote by $P_N$ the Littlewood-Paley projection given by
	$$ \widehat{P_N f}(\xi) = \psi_N(\xi) \widehat{f}(\xi). $$ Unless otherwise noted, Fourier transformation is taken with respect to the $x$ variable to define $P_N$. 
	
The space $\dot{Z}^s$ is the closure of $U^2_S \cap C(\R,H^{1,1}(\R^2))$ in the norm
\begin{equation}
	\label{Zs.norm}
		\norm[\dot{Z}^s]{u} 
			=	\left( \sum_{N} N^{2s} \norm[U^2_S]{P_N u}^2 \right)^\frac12,
\end{equation}
where the sum is over all dyadic integers $N$. Similarly, $\dot{Y}^s$ is the closure of functions $v \in V^2_{-,{\rc},S} \cap C(\R,H^{1,1}(\R^2))$ with
\begin{equation}
	\label{Ys.norm}
		\norm[\dot{Y}^s]{u} 
			=	\left( \sum_{N} N^{2s} \norm[V^2_S]{P_N u}^2 \right)^\frac12	.
\end{equation}

By initially considering functions in $C(\R,H^{1,1}(\R^2))$ we will be able to analyze the bilinear map \eqref{KPII-5-IT} and extend it by density to a bounded bilinear map from $\dot{Z}^{-\frac12} \times \dot{Z}^{-\frac12}$ to $\dot{Z}^{-\frac12}$. 
Here we will take $s=-\frac12$ corresponding to the choice of $\dot{H}^{-\frac12,0}(\R^2)$ as the space of initial data. Owing to the continuous embedding $U^2_S \subset V^2_S$, we also have $\dot{Z}^{s} \subset \dot{Y}^{s}$ with continuous embedding. Since $\dot{Y}^{-\frac12}$ and $\dot{Z}^{-\frac12}$ are closures of a set of continuous functions in norms that dominate the $C(\R, \dot{H}^{-\frac12,0})$ norm, it follows that $\dot{Y}^{-\frac12}$ and $\dot{Z}^{-\frac12}$ functions lie in $C(\R, \dot{H}^{-\frac12,0})$ with 
\begin{align}
	\label{Y.cont}
	\norm[C(\R,\dot{H}^{-\frac12})]{u} &\lesssim \norm[\dot{Y}^{-\frac12}]{u}
\intertext{and}
	\label{Z.cont}
	\norm[C(\R,\dot{H}^{-\frac12})]{u} &\lesssim \norm[\dot{Z}^{-\frac12}]{u}	
\end{align}
respectively.

Our scattering result depends on the following fact.

\begin{lemma}
	\label{lemma.Y.scat}
	Suppose that $u \in \dot{Y}^{-\frac12}$. Then 
	$$ \varphi_\pm = \lim_{t \to \pm \infty} S(-t)u(t) $$
	exists in $\dot{H}^{-\frac12,0}(\R^2)$.
\end{lemma}

\begin{proof}
	From the bound $\norm[L^2]{S(-t)f(t)} \lesssim \norm[V^2_S]{f(t)}$, uniform in $t$, we have
		$\norm[L^2]{P_N S(-t) u(t)} \lesssim \norm[V^2_S]{P_N u}$ and hence
		\begin{equation} 
		\label{St.ut.bd}
		\sum_N N^{-1} \norm[L^2]{P_N S(-t) u(t)}^2 \lesssim \norm[\dot{Y}^{-\frac12}]{u}^2 
		\end{equation}
		uniformly in $t$. Let $$u_{\pm,N} = \lim_{t \to \pm \infty} P_N S(-t) u(t)$$ where the limit is taken in $L^2(\R^2)$ and exists owing to the fact that $P_N u \in V^2_S$. From the bound \eqref{St.ut.bd} we may conclude that $u_\pm = \sum_N u_{\pm, N} \in \dot{H}^{-\frac12,0}(\R^2)$. 
		
		For any $\eps>0$ we can find a dyadic integer $L$ so that 
		$$ \left( \sum_{N < L^{-1}}+\sum_{N> L} \right) N^{-1} \norm[L^2]{P_N S(-t) u(t)}^2 < \frac{\eps}{4}$$
		uniformly in $t$, and hence, also 
		$$ \left( \sum_{N < L^{-1}}+\sum_{N> L} \right) N^{-1} \norm[L^2]{u_{\pm,N}}^2 \lesssim  \frac{\eps}{4}.$$ 
		On the other hand, for all dyadic integers $N$ with $L^{-1} \leq N \leq L$, we can find $T$ with $\pm T > 0$ and large so that 
		$$ \sum_{L^{-1} \leq N \leq L} N^{-1} \norm[L^2]{P_N S(-t)u(t) - u_{\pm,N}}^2 < \frac{\eps}{2}. $$
		It now follows that $S(-t)u(t) \to u_\pm$ in $\dot{H}^{-\frac12,0}(\R^2)$ as $t \to \pm \infty$. 
\end{proof}

\subsection{Modulation Decomposition}
\label{subsec:module}

It will also be useful to decompose functions with respect to their modulation. To define the decomposition, first denote by $Q_M$ the Littlewood-Paley multipliers in the $\tau$ variable (the Fourier variable corresponding to $t$) by
	$$ \widehat{Q_M f} =\psi_M(\tau) \widehat{f}(\tau) $$
	where $\widehat{\, \cdot \,}$ denotes the Fourier transform in $t$ and $M$ is a dyadic integer. We now set
	\begin{equation}
		\label{QSM}
			Q^S_M = S(t) Q_M S(-t) 
	\end{equation}
	so that 
	$$ \widehat{Q^S_M f}(\tau,\xi,\zeta) = \psi_M(\tau - p(\xi,\zeta)) \widehat{f}(\tau,\xi,\zeta)$$
	where now $\widehat{\, \cdot \,}$  denotes the Fourier transform in $(t,x,y)$ and $S(t)$ has symbol $e^{itp(\xi,\eta)}$ (see \eqref{S.symbol}).  For dyadic integers $M$, we set
	\begin{align}
		\label{QS.small}
		Q_{S< M}	&=	\sum_{N < M} Q^S_N\\
		\label{QS.large}
		Q_{S \geq M}&=	\sum_{N \geq M} Q^S_N
	\end{align}
	which project onto states with low or high modulation. We will use this decomposition to prove estimates on the bilinear form \eqref{KPII-5-IT}.
	
	We recall from \cite[Corollary 2.18]{HHK09} the following estimates:
	\begin{align}
		\label{Q.big1}
		\norm[L^2(\R^3)]{Q_{\geq M}^S v} & \lesssim M^{-\frac12} \norm[V^2_S]{v}\\
		\label{Q.Up}
		\norm[U^p_S]{Q_{<M}^S v} & \lesssim \norm[U^p_S]{v} &
		\norm[U^p_S]{Q_{\geq M}^S v} & \lesssim \norm[U^p_S]{v}\\
		\label{Q.Vp}
		\norm[V^p_S]{Q_{<M}^S v} & \lesssim \norm[V^p_S]{v} &
		\norm[V^p_S]{Q_{\geq M}^S v} & \lesssim \norm[V^p_S]{v}
	\end{align}
	These carry over immediately from the third-order linear evolution in \cite{HHK09} to the fifth-order evolution considered here since the proofs of the estimates only use the relation \eqref{QSM} together with the definition of $U^p_S$ and $V^p_S$ with respect to $S(t)$.

\section{Strichartz Estimates}

In this section we prove estimates on the linear evolution $S(t)$ 
which imply,  by the transfer principle for $U^p$ spaces (see \cite[Proposition 2.19]{HHK09}) and interpolation, (see \cite[Proposition 2.20]{HHK09}),  linear and bilinear estimates for functions in $U^p_S$ and $V^p_S$. 

For the linear problem \eqref{KPII-5-linear},  the solution operator $S(t)$ acts as a Fourier multiplier with symbol $\exp(itp(\xi,\eta))$ where
\begin{equation}
\label{S.symbol}
p(\xi,\eta) = -\beta \xi^5 +\alpha \xi^3 - \xi^{-1} \eta^2. 
\end{equation}

Saut and Tzvetkov \cite[Lemma 2]{ST99} (see also \cite[Theorem 4.2]{BS99} proved that for any $p,r$ with $\frac1p+\frac1r = \frac12$, the estimate
\begin{equation}
    \label{Strichartz:ST}
    \norm[L^p_t(\R, L^r(\R^2)]{|D|^{1/p} S(t) f} \lesssim \norm[L^2]{f}
\end{equation}
holds. This estimate gives good decay for high-frequency modes (where the $\xi^5$ in $p(\xi,\eta)$ dominates) but for lower frequency modes (where the $\xi^3$ term in $p(\xi,\eta)$ dominates) we can get a finer estimate. 

\begin{lemma}
\label{lemma:Strichartz}
Suppose that $\alpha>0$ and $2 \leq q,r < \infty$ with $q^{-1} + r^{-1} = \frac12$,  and let
$$ \delta(r) = 1- \frac{2}{r} = \frac{2}{q}. $$
For any $\varphi \in L^2$,  the estimates
\begin{align}
\label{S1.small}
\norm[L^q_t(\R,L^r(\R^2))]{S(t) P_{\lesssim 1} \varphi}
	&	\lesssim_{\,  q,r} \norm[L^2(\R^2)]{P_{\lesssim 1} \varphi}\\
\intertext{and,  for $N \gtrsim 1$,}
\label{S1.large}
\norm[L^q_t(\R,L^r(\R^2))]{S(t) P_{N} \varphi}
	&	\lesssim _{\, q,r} N^{-\delta(r)/2} \norm[L^2(\R^2)]{P_N \varphi}.
\end{align}
\end{lemma}

\begin{proof}[Sketch of proof]
    The estimate \eqref{S1.large} is an immediate consequence of \eqref{Strichartz:ST}. To obtain \eqref{S1.small}, we take $\beta=-1$ for simplicity and begin by computing the integral kernel $G_N(x,y,t)$ for the operator $S(t)P_N$.  We have 
    \begin{equation}
        \label{S.GN.int}
            G_N(x,y,t) = 
	           \frac{1}{(2\pi)^{\frac32} \sqrt{t}} 
		      \int |\xi|^{\frac12} \psi_N(\xi) e^{i\phi(\xi;x,y,t)} \,  d\xi
    \end{equation}
    where, for $\pm \xi>0$
    \begin{equation}
        \label{S.GN.phase}
        \phi(\xi;x,y,t) = \pm \frac{\pi}{4} 
                    + \left( x+ \frac{y^2}{4t} \right) \xi
	               +\xi^5 + \alpha \xi^3
    \end{equation}
    (compare Saut \cite{Saut93} and Ben Artzi-Saut \cite{BS99} where similar estimates are carried out for the solution operator of the third-order linear equation associated to the third-order KP II equation). Observing that $|\phi''(\xi)| \gtrsim |\xi|$ for $N \lesssim 1$, we can use the Van der Corput lemma (see, for example, \cite[Corollary in \S VIII.1.2]{Stein93}) to conclude that
    $$ |G_N(x,y,t)| \lesssim t^{-1}, \quad N \lesssim 1.$$ This implies that 
    $$ \norm[L^\infty]{S(t)P_N \varphi} \lesssim t^{-1}\norm[L^1]{P_N \varphi}$$ so that, by interpolation with the trivial $L^2 \to L^2$ estimate, 
    $$ \norm[L^{p'}]{S(t)P_N\varphi} \lesssim t^{1-\frac2p} \norm[L^p]{P_N \varphi}$$
    for $t>0$ and $p \in [1,2]$. It now follows by standard methods (see for example \cite[\S 2.3]{Tao06} that
    $$ \norm[L^q(R,L^p(\R^2)]{S(t)P_N \varphi} \lesssim \norm[L^2(\R^2)]{P_N \varphi}, \quad N \lesssim 1. $$ 
    Using the inequality
    $$
        \norm[L^q_t(\R,L^r(\R^2))]{u}
	       \lesssim_{\, q, r} 
	       	\left( 
	       		\sum_N \norm[L^q_t(\R,L^r(\R^2))]{P_N u}^2
	       	\right)^{\frac12}
$$
(see for example \cite[Section 3.1]{CKSTT08} or \cite[Exercise A.14]{Tao06}) we obtain \eqref{S1.small}.
    
\end{proof}

\begin{remark}
From \eqref{S1.large} we easily obtain
$$
\norm[L^q_t(\R,L^r(\R^2))]{S(t) P_{\gtrsim 1} \varphi}
\lesssim _{\, q,r} \norm[L^2(\R^2)]{P_{\gtrsim 1} \varphi}
$$
from which it follows that
\begin{equation}
\label{S1.complete}
\norm[L^q_t(\R,L^r(\R^2))]{S(t) \varphi} 
\lesssim_{\, q,r} \norm[L^2(\R^2)]{\varphi}.
\end{equation}
The estimate \eqref{S1.complete} also holds when $S(t)$ is the solution operator for the linear equation $v_t + v_{xxx} + \diff_x^{-1} v_{yy}=0$ reassociated to the third-order KP II equation \eqref{KPII-3} (see \cite{Saut93}, Proposition 2.3).
\end{remark}

Next, we prove the following bilinear estimate.  In what follows, we write
$\zeta$ for $(\xi,\eta)$,  the Fourier variables corresponding to $(x,y)$.

\begin{lemma}
\label{lemma:BS-1}
For $N_1 \lesssim N_2$, the estimate
\begin{equation}
\label{BS-1}
\norm[L^2(\R^3)]{
	\left( S(t) P_{N_1} \varphi_1 \right)	\cdot
	\left( S(t) P_{N_2} \varphi_2 \right)
	}
	\lesssim \left( \frac{N_1}{N_2} \right)^{\frac12}
	\norm[L^2(\R^2)]{P_{N_1} \varphi_1}\norm[L^2(\R^2)]{P_{N_2} \varphi_2}
\end{equation}
holds.
\end{lemma}

\begin{proof}
First, if $N_1 \sim N_2$,  we obtain \eqref{BS-1} by applying H\"{o}lder's inequality and using \eqref{S1.complete} with $q=r=4$.  Henceforth
we assume $N_1 \ll N_2$. 

Second, we note the resonance identity (e.g.  \cite[equation (7)]{ST00})
\begin{equation}
\label{resonance}
p(\zeta_1) + p(\zeta-\zeta_1) - p(\zeta) =
\nu + \frac{(\xi_1 \eta - \xi \eta_1)^2}{\xi(\xi-\xi_1)\xi_1}
\end{equation}
where $p$ is given by \eqref{S.symbol} and 
\begin{equation}
\label{nu}
\nu = \xi(\xi-\xi_1) \xi_1 \left( 5(\xi^2-\xi \xi_1 +\xi_1^2) + 3\alpha \right)
\end{equation}
has the same sign as the second right-hand term in \eqref{resonance} provided $\alpha>0$.
From these equations we have the identity
\begin{equation}
\label{resonance.est}
\left| 
	p(\zeta_1) + p(\zeta-\zeta_1) - p(\zeta) - \nu
\right|
=
\frac{|\xi_1\eta - \xi \eta_1|^2}
	{|\xi| |\xi-\xi_1| |\xi_1|}
\end{equation}
and the inequality
\begin{equation}
\label{nu.range}
0 \leq |\nu| \leq |p(\zeta) + p(\zeta-\zeta_1) - p(\zeta)|.
\end{equation}

Third,  from the Fourier representation
\begin{multline*}
\left( S(t) P_{N_1} \varphi_1 \right) \cdot 
\left( S(t) P_{N_2} \varphi_2 \right) (x,y) =\\
\frac{1}{(2\pi)^2}
	\iint e^{i\phi(t,x,y,\zeta_1, \zeta_2)} 
		\psi_{N_1}(\xi_1) \psi_{N_2}(\xi_2)
		 \widehat{\varphi_1}(\zeta_1) 
									\widehat{\varphi_2}(\zeta_2) \
		\,  d\zeta_1  \,  d\zeta_2,
\end{multline*}
where
$$
\phi(t,x,y,\zeta_1,\zeta_2) =
	x(\xi_1 +\xi_2) + y (\eta_1 + \eta_2) -t(p(\zeta_1) + p(\zeta_2)),
$$
we see that $\left( S(t) P_{N_1} \varphi_1 \right)	\cdot
\left( S(t) P_{N_2} \varphi_2 \right)$ is the inverse Fourier transform
(in $\lambda$, $\xi$, and $\eta$) of
$$ I(\lambda,\xi,\eta)=
	\int \delta_0(\lambda-p(\zeta_1) - p(\zeta)) 
			\psi_{N_1}(\xi_1) \psi_{N_2}(\xi-\xi_1)
			\widehat{\varphi_1}(\zeta_1)
			\widehat{\varphi_2}(\zeta-\zeta_1)
	\,  d\zeta_1.
$$
By the Plancherel theorem,  it suffices to estimate the $L^2$ norm
of $I(\lambda,\xi,\eta)$,  given by
$$
\sup_{g \in L^2(\R^3): \,\, \norm[L^2]{g}=1}
	\left|
		\int g(\lambda,\zeta) I(\lambda,\xi,\eta) \, d\lambda \, d\zeta
	\right|.
$$
A short computation shows that 
\begin{multline*}
\int g(\lambda,\zeta) I(\lambda,\xi,\eta) \,  d\lambda \, d\zeta=\\
\int g(p(\zeta_1) + p(\zeta-\zeta_1),\zeta)
	\psi_{N_1}(\zeta_1) \psi_{N_2}(\zeta-\zeta_1) 
	\widehat{\varphi_1}(\zeta_1)  
	\widehat{\varphi_2}(\zeta-\zeta_1) 
\,  d\zeta_1 \, d\zeta.
\end{multline*}
We make the change of variables
$$ u=p(\zeta_1) + p(\zeta-\zeta_1),
	\quad
	v=\zeta,
	\quad
	w=\nu
$$
where $J=|\diff(u,v,w)/\diff(\zeta_1,\zeta)|$.  Assuming that $J \neq 0$ it then suffices to bound
\begin{multline*}
\iiint g(u,v) \frac{\psi_{N_1} \psi_{N_2} \widehat{\varphi_1} \widehat{\varphi_2}}{J} \,  du\,  dv\,  dw
\end{multline*}
which is itself bounded in modulus by
$$
\iint |g(u,v)| \left( \int \frac{|\tilde{\psi}_{N_1} \tilde{\psi}_{N_2}|}{J} \, dw' \right)^\frac12
					\left( \int \frac{|\psi_{N_1}\psi_{N_2} \widehat{\varphi_1} \widehat{\varphi_2}|^2}{J} \,  dw' \right)^\frac12 \,  du \,  dv
$$
where $\tilde{\psi}_{N_1}$ (resp.\ $\tilde{\psi}_{N_2}$) is identically $1$ on the support of $\psi_{N_1}$ (resp.\ $\psi_{N_2}$) and  $|\xi_1| \sim N_1$ (resp.\ $|\xi-\xi_1| \sim N_2$).  This integral is,  in turn,  bounded by
\begin{equation}
\label{bs.to.estimate}
 \left( 
			\int \frac{\tilde{\psi}_{N_1} \tilde{\psi}_{N_2}}{J} \, dw' 
	\right)^\frac12
	 \norm[L^2]{P_{N_1} \varphi_1} \norm[L^2]{P_{N_2}\varphi_2} 
\end{equation}
since $\norm[L^2]{g}=1$.  

It now suffices to show that $J>0$ and obtain effective bounds on $J$.  A short computation shows that
$$ J = \left| \frac{\diff u}{\diff \eta_1} \cdot \frac{\diff \nu}{\diff \xi_1} \right|. $$
We have from \eqref{resonance} and \eqref{resonance.est} that
\begin{equation}
\label{du.deta}
\left| \frac{\diff u}{\diff \eta_1} \right| = 2 \left| \frac{\xi_1 \eta - \xi \eta_1}{\xi(\xi-\xi_1)\xi_1} \right| \,  |\xi| =
2\frac{|u-p(\zeta)-\nu|^{\frac12}|\xi|}{|\xi|^\frac12 |\xi-\xi_1|^\frac12 |\xi_1|^\frac12}.
\end{equation}
From  \eqref{nu} it is easy to deduce that for $|\xi| \gg |\xi_1|$,
\begin{equation}
\label{nu.est}
|\nu|	\sim
\begin{cases}
|\xi_1| \, |\xi-\xi_1|^4,		&	N_2 \gtrsim 1\\
|\xi_1| \, |\xi-\xi_1|^2,		&	N_2 \lesssim 1
\end{cases}
\end{equation}
and
\begin{equation}
\label{dnu.est}
\left| \frac{\diff \nu}{\diff \xi_1} \right|
\sim
\begin{cases}
|\xi-\xi_1|^4,	&	N_2 \gtrsim 1\\
|\xi-\xi_1|^2,	&	N_2	\lesssim 1\\
\end{cases}
\end{equation}
From \eqref{du.deta},  \eqref{nu.est},  and \eqref{dnu.est},  we conclude that,
for $N_2 \gtrsim 1$,
\begin{align}
\label{J.highN}
J	
	&\gtrsim 	
		\frac	{|u-p(\zeta)-\nu|^\frac12 |\xi| \, |\xi-\xi_1|^4}
				{|\xi|^\frac12 |\xi-\xi_1|^\frac12 |\xi_1|^\frac12}\\[5pt]
\nonumber
	&\gtrsim
		\frac{|u-p(\zeta)-\nu|^\frac12 |\xi|\, |\nu|^\frac12 |\xi-\xi_1|^{2}}
			{|\xi|^\frac12 |\xi-\xi_1|^\frac12 |\xi_1|}\\
\nonumber
	&\gtrsim
		|u-p(\zeta)-\nu|^\frac12 \, |\nu|^\frac12 \cdot \frac{|\xi-\xi_1|^2}{|\xi_1|}
\end{align}
From the same identity and estimates we conclude that, for $N_2 \lesssim 1$,
\begin{align}
\label{J.lowN}
J	
	&\gtrsim 	
		\frac	{|u-p(\zeta)-\nu|^\frac12 |\xi| \, |\xi-\xi_1|^2}
				{|\xi|^\frac12 |\xi-\xi_1|^\frac12 |\xi_1|^\frac12}\\[5pt]
\nonumber
	&\gtrsim
		\frac	{|u-p(\zeta)-\nu|^\frac12 |\xi| \,  |\nu|^{\frac12} |\xi-\xi_1|}
				{|\xi|^\frac12 \, |\xi-\xi_1|^\frac12 \, |\xi_1|}\\[5pt]
\nonumber
	&\gtrsim
		|u-p(\zeta)-\nu|^\frac12 \, |\nu|^\frac12 \cdot \frac{|\xi-\xi_1|}{|\xi_1|}
\end{align}
To bound the integral factor in \eqref{bs.to.estimate},  we now use \eqref{J.highN} or \eqref{J.lowN} together with the estimate 
\begin{equation}
\label{bs.to.conclude}
\int \frac{1}{|u-p(\zeta)-w'|^\frac12 \, |w'|^\frac12} \,  dw'
\lesssim 1.
\end{equation}
Here  $\nu = w'$,  and the estimate follows from the facts that 
$$0 \leq |w'| \leq |u-p(\zeta)|$$
in virtue of \eqref{nu.range}, and the trivial estimate
$$\int_0^a \frac{1}{x^\frac12 (a-x)^\frac12} \,  dx = \bigO{1}$$ 
uniformly in $a>0$.  We now conclude that
\begin{multline}
\label{bs.estimate}
\norm[L^2(\R^3)]{\left
								(S(t) P_{N_1} \varphi_1\right) \cdot 
								\left(S(t) P_{N_2} \varphi_2 \right)}	\\
\lesssim 
	\begin{cases}
	\left(\dfrac{N_1}{N_2^2}\right)^\frac12
			 \norm[L^2]{P_{N_1}\varphi_1} 
			 \norm[L^2]{P_{N_2}\varphi_2}
		&	N_2 \gtrsim 1,\\
	\\
	\left(\dfrac{N_1}{N_2}\right)^\frac12 
			\norm[L^2]{P_{N_1}\varphi_1} 
			\norm[L^2]{P_{N_2}\varphi_2}
		&	N_2 \lesssim 1.
	\end{cases}
\end{multline}
This implies \eqref{BS-1}.
\end{proof}

Owing to the Strichartz estimate \eqref{S1.complete} and the bilinear Strichartz estimate \eqref{BS-1},  we can use Proposition \ref{prop:transfer} and interpolation estimates from \cite[Proposition 2.20]{HHK09} to deduce the following.

\begin{proposition}
\label{prop:estimates}
Suppose that $\alpha>0$. 
For any $2 \leq q,r  < \infty$ with 
$$q^{-1} + r^{-1} = \frac12,$$
the estimates
\begin{align}
\label{UP.S1}
\norm[L^q_t(\R,L^r(\R^2))]{u}		&\lesssim \norm[U^q_S]{u}\\
\intertext{and}
\label{VP.S1}
\norm[L^4(\R^3)]{u}		&\lesssim \norm[V^p_{-,S}]{u},	&1 \leq p < 4
\end{align}
hold. 
Moreover
\begin{align}
\label{Up.Bi}
\norm[L^2(\R^3)]{P_{N_1} u P_{N_2} v}	
	&\lesssim	\left( \frac{N_1}{N_2} \right)^\frac12 
						\norm[U^2_S]{P_{N_1}u}
						\norm[U^2_S]{P_{N_2}v}
\end{align}
and
\begin{multline}
\label{Vp.Bi}
\norm[L^2(\R^3)]{P_{N_1} u P_{N_2} v}	\\
	\lesssim	\left( \frac{N_1}{N_2} \right)^\frac12 
					\left( 1 + \log\left(\frac{N_2}{N_1}\right)\right)^2\
						\norm[V^2_S]{P_{N_1}u}
						\norm[V^2_S]{P_{N_2}v}
\end{multline}
\end{proposition}

\begin{proof}
We only comment on the proofs of \eqref{VP.S1}--\eqref{Vp.Bi} since the others are immediate consequences of Proposition \ref{prop:transfer} and the previous estimates.  Exactly as in the proof of \cite[Corollary 2.21]{HHK09},  \eqref{VP.S1} follows from \eqref{UP.S1} and the continuity of the embedding $V^p_{-,\mathrm{rc}}$ into $U^q$ for $1 \leq p < q$ (see \cite[Corollary 2.6]{HHK09}). The estimate \eqref{Up.Bi} follows from \eqref{BS-1} and Proposition \ref{prop:transfer} for $n=2$.
The estimate \eqref{Vp.Bi} follows from \eqref{Up.Bi} by the interpolation argument in the proof of \cite[Corollary 2.21]{HHK09}.
\end{proof}

\section{Estimates on the Bilinear Form}

In what follows we define, for $T>0$ and $u,v \in C(\R,H^{1,1}(\R^2))$
\begin{equation}
	\label{KPII-5-IT-bis}
	I_T(u,v) = \frac12 \int_0^t \mathbbm{1}_T(s) S(t-s) (u(s)v(s))_x \, ds
\end{equation}
where $\mathbbm{1}_T$ is the characteristic function of $[0,T]$. We also define
\begin{equation}
	\label{KPII-5-IT-infty}
	I_\infty(u,v) = \frac12 \int_0^t  S(t-s) (u(s)v(s))_x \, ds
\end{equation}

\subsection{Fundamental Estimate}

First, we will prove (modulo the estimates proved in Lemma \ref{lemma:J1J2} in what follows):
\begin{proposition}
	\label{prop:IT.est}
	The map $I_T$ extends to a bounded bilinear map from $\dot{Z}^{-\frac12} \times \dot{Z}^{-\frac12}$ to $\dot{Z}^{-\frac12}$ with
	\begin{equation}
		\label{IT.bdYY}
			\norm[\dot{Z}^{-\frac12}]{I_T(u,v)}
			\leq
			C 	\norm[\dot{Y}^{-\frac12}]{u}
				\norm[\dot{Y}^{-\frac12}]{v}
	\end{equation}
	where $C$ is independent of $T$.
\end{proposition}

\begin{remark}
	\label{rem:IT}
	It follows from the continuous embedding $\dot{Z}^{-\frac12} \hookrightarrow \dot{Y}^{-\frac12}$ that
	\begin{align}
		\label{IT.bdZZ}
		\norm[\dot{Z}^{-\frac12}]{I_T(u,v)}
		\leq C' \norm[\dot{Z}^{-\frac12}]{u}
			\norm[\dot{Z}^{-\frac12}]{v}	
	\intertext{and}
		\label{IT.bdYYbis}
		\norm[\dot{Y}^{-\frac12}]{I_T(u,v)}
		\leq C'' \norm[\dot{Y}^{-\frac12}]{u}
				\norm[\dot{Y}^{-\frac12}]{v}
	\end{align}
	with $C'$, $C''$ independent of $T$. 
\end{remark}

\begin{proof}
	The proof follows closely the proof of \cite[Theorem 3.2]{HHK09}:  we sketch the details for the reader's convenience since the same strategy of proof will be used later to prove further estimates on the bilinear form. By the Littlewood-Paley decomposition, we have
	$$  I_T(u,v) = \sum_{N_1,N_2}  I_T (u_{N_1},v_{N_2}) $$
	where  $u_{N_1} = P_{N_1} u$ and $v_{N_2} = P_{N_2} v$. By bilinearity and symmetry it suffices to show that
	\begin{align}
		\label{IT.J1}
		\norm[\dot{Z}^{-\frac12}]{\sum_{N_2} \sum_{N_1 \ll N_2} I_T(u_{N_1},v_{N_2})}
			&	\lesssim 
					C 	\norm[\dot{Y}^{-\frac12}]{u}
						\norm[\dot{Y}^{-\frac12}]{v}
		\intertext{and}
		\label{IT.J2}
		\norm[\dot{Z}^{-\frac12}]{\sum_{N_2} \sum_{N_1 \sim N_2} I_T(u_{N_1},v_{N_2})}
			&	\lesssim 
					C 	\norm[\dot{Y}^{-\frac12}]{u}
						\norm[\dot{Y}^{-\frac12}]{v}
	\end{align}
	for a constant $C$ independent of $T$. 
	To estimate the left-hand sides of \eqref{IT.J1} and \eqref{IT.J2}, we exploit the facts that 
	\begin{align}
		\label{IT.J1bis}
		\norm[\dot{Z}^{-\frac12}]{\sum_{N_1 \ll N_2} I_T(u_{N_1},v_{N_2})}^2
		& = \sum_{N_3} N_3^{-1} \norm[U^2_S]{P_{N_3} \sum_{N_1 \ll N_2}I_T(u_{N_1}, v_{N_2})}^2
		\intertext{and}
		\label{IT.J2bis}
		\norm[\dot{Z}^{-\frac12}]{\sum_{N_1 \sim N_2} I_T(u_{N_1},v_{N_2})}^2
		& = \sum_{N_3} N_3^{-1} \norm[U^2_S]{P_{N_3} \sum_{N_1 \sim  N_2}I_T(u_{N_1}, v_{N_2})}^2
	\end{align} 
	and use duality to obtain summable estimates on the $U^2_S$ norms. To bound \eqref{IT.J1} and \eqref{IT.J2} in terms of $\norm[\dot{Z}^{-\frac12}]{u}$ and $\norm[\dot{Z}^{-\frac12}]{v}$ we exploit the respective estimates \eqref{pre-bilinear-1} and \eqref{pre-bilinear-2} in Lemma \ref{lemma:J1J2} below. 
	
	First consider \eqref{IT.J1}. For $N_1 \ll N_2$ we we obtain nonzero contributions when $N_3 \sim N_2$. By \eqref{Ups.norm.dual}, we have
	\begin{align*}
		\Bigl\| 
				P_{N_3} 
				\sum_{N_1 \ll N_2} & I_T(u_{N_1}, v_{N_2}) 
		\Bigr\|_{U^2_S}\\
			&\lesssim	
				\sup_{\substack{\diff_t w \in C_0^\infty(\R^3),\\ \norm[V^2]{w}\leq 1}}
					\left|
						\sum_{N_1 \ll N_2} \int \langle  u_{N_1}(t) v_{N_2}(t), \diff_x P_{N_3} S(-t)w \rangle_{L^2} \, dt 
					\right|\\
			&=  \sup_{\substack{\diff_t w \in C_0^\infty(\R^3),\\ \norm[V^2]{w}\leq 1}}
					\left|
						\sum_{N_1 \ll N_2} \int \langle  u_{N_1}(t) v_{N_2}(t), \diff_x P_{N_3} w \rangle_{L^2} \, dt 
					\right|\\
			&\lesssim  
					\left( 
							\sum_{N_1 \ll N_2} N_1^{-1}\norm[V^2_S]{u_{N_1}}^2 
					\right)^\frac12
			\norm[V^2_S]{v_{N_2}}  
	\end{align*} 
	where in the last step we used \eqref{pre-bilinear-1}, the fact that  $N_2 \sim N_3$ and the bound $\norm[L^2]{\diff_x w_{N_3}} \lesssim N_3 \norm[L^2]{w_{N_3}}$. Using this estimate in the right-hand side of \eqref{IT.J1} and using the fact that $N_2 \sim N_3$ we recover
	\begin{equation}
		\label{IT.J1.Final}
		\norm[\dot{Z}^{-\frac12}]{\sum_{N_1 \ll N_2} I_T(u_{N_1},v_{N_2})}^2
			\lesssim	
				\norm[\dot{Y}^{-\frac12}]{u}^2 \norm[\dot{Y}^{-\frac12}]{v}^2
	\end{equation}
	with a constant independent of $T$. 
	
	Next, consider \eqref{IT.J2}. As $N_1 \sim N_2$ we obtain nonzero contributions for $N_3 \lesssim N_2$.  We may estimate
	\begin{align}
		\label{IT.J2.Minkowski}
		\norm[\dot{Z}^{-\frac12}]{\sum_{N_2} \sum_{N_1 \sim N_2} I_T(u_{N_1},v_{N_2}) }
			\lesssim \sum_{N_2} \sum_{N_1 \sim N_2} \norm[\dot{Z}^{-\frac12}]{ I_T(u_{N_1},v_{N_2})}
	\end{align}
	and then estimate
	\begin{align}
		\label{IT.J2.Details}
		\norm[\dot{Z}^{-\frac12}]{I_T(u_{N_1},v_{N_2})}^2
			&=	\sum_{N_3 \lesssim N_2} 
					N_3^{-1} \norm[U^2_S]{P_{N_3} I_T(u_{N_1},v_{N_2})}^2\\
			&\lesssim   \sum_{N_3 \lesssim N_2} 
						N_3 \sup_{\substack{\diff_t w \in C_0^\infty \\\norm[V^2_S]{ w} \leq 1}}	
							\left| \int \langle  u_{N_1} v_{N_3}, P_{N_3} w\rangle \, dt \right|^2
			\nonumber \\
			&\lesssim N_1^{-1} \norm[V^2_S]{u_{N_1}}^2 N_2^{-1} \norm[V^2_S]{v_{N_2}}^2
			\nonumber 
	\end{align}
	where in the last step we used \eqref{pre-bilinear-2}. Using \eqref{IT.J2.Details} in \eqref{IT.J2.Minkowski} we conclude that
	\begin{equation}
		\label{IT.J2.Final}
		\norm[\dot{Z}^{-\frac12}]{\sum_{N_2} \sum_{N_1 \sim N_2} I_T(u_{N_1},v_{N_2})}
		\lesssim \norm[\dot{Y}^{-\frac12}]{u}^2 \norm[\dot{Y}^{-\frac12}]{v}^2
	\end{equation}
	again with an implied constant independent of $T$. 
	\end{proof}

Now we prove the termwise estimates \eqref{pre-bilinear-1} and \eqref{pre-bilinear-2} used n the proof of Proposition \ref{prop:IT.est}. Denote by $A_N$ the annular region $N/2 \leq |\xi| \leq 2N$. 
\begin{lemma}
	\label{lemma:J1J2}
	Let $N_1$, $N_2$ and $N_3$ be dyadic integers, and suppose
	that $u_{N_1}$, $v_{N_2}$, and $w_{N_3}$ have Fourier supports respectively in $A_{N_1}$, $A_{N_2}$ and $A_{N_3}$. The following estimates hold:
	\begin{enumerate}[(a)]
		\item	If $N_2 \sim N_3$ then
				\begin{multline}
				\label{pre-bilinear-1}
					\left|
						\sum_{N_1 \lesssim N_2} \int_0^T \int_{\R^2}
							u_{N_1} v_{N_2}  w_{N_3}  \,  dx \,  dy \, dt
					\right|	\\
				\lesssim 
				\left( 
					\sum_{N_1 \lesssim N_2} 
						N_1^{-1} \norm[V^2_S]{u_{N_1}}^2
				\right)^{\frac12}
				N_2^{-\frac12} \norm[V^2_S]{v_{N_2}}
				N_3^{-\frac12} \norm[V^2_S]{w_{N_3}}
			\end{multline}
	\item	If $N_1 \sim N_2$ then
\begin{multline}
\label{pre-bilinear-2}
\left(
\sum_{N_3 \lesssim N_2}
	N_3
		\sup_{\substack{\diff_t w \in C_0^\infty \\ \norm[V^2_S]{w_{N_3}}\leq 1}}
			\left|
				 \int_0^T \int_{\R^2}
					u_{N_1} v_{N_2}  w_{N_3}  \,  dx \,  dy \, dt
			\right|^2
\right)^{\frac12}\\
\lesssim 
N_1^{-\frac12}\norm[V^2_S]{u_{N_1}}
N_2^{-\frac12} \norm[V^2_S]{v_{N_2}}
\end{multline}
\end{enumerate}
with implied constants independent of $T$. 
\end{lemma}

\begin{proof}
	The proof of these estimates is identical to the proof of \cite[Proposition 3.1]{HHK09} but we sketch the proof since we will use the same strategy again later to prove stronger estimates on the bilinear form. To bound integrals of the form
	$$ I(u_1,v_2,w_3) = \int_0^T \langle u_1 v_2, w_3 \rangle_{L^2} \, dt $$
	we first let $\widetilde{u_1} = \mathbbm{1}_{[0,T)}u_1$, $\widetilde{v_2} = \mathbbm{1}_{[0,T)} v_2$, $\widetilde{w_3} = \mathbbm{1}_{[0,T)} w_3$ and then introduce the low-modulation/high modulation decomposition $I = Q^S_{<M} + Q_{\geq M}$ (see \eqref{QS.small}--\eqref{QS.large}) to decompose $I(u_1,v_2,w_3)$ into eight terms of the form 
	$$ \int_{\R} \langle Q_1^S \widetilde{u_1} Q_2^S \widetilde{v_2} , Q_3^S \widetilde{w_3} \rangle_{L^2} \, dt $$
	where $Q_i^S$ is either $Q^S_{<M}$ or $Q^S_{\geq M}$. By Fourier analysis (in the three variables $(t,x,y)$ with Fourier variables $(\tau,\xi,\eta)$) we have
	$$		
		\int_{\R} \langle Q_1^S \widetilde{u_1} Q_2^S \widetilde{v_2} , Q_3^S \widetilde{w_3} \rangle_{L^2} \, dt = \\
			\left(\widehat{Q_1^S \widetilde{u_1}} *
					\widehat{Q_2^S \widetilde{v_2} } *
					\widehat{Q_3^S \widetilde{w_3}}
			\right)(0)
	$$
	where $*$ denotes convolution in the variables $(\tau,\xi,\eta)$. Setting $\mu_i = (\xi_i,\eta_i,\tau_i)$ for $i=1,2,3$ the right-hand side may be written
	$$ \int_{\mu_1+\mu_2+\mu_3=0} 
			\widehat{Q_1^S \widetilde{u_1}}(\mu_1)
			\widehat{Q_2^S \widetilde{v_2} }(\mu_2)
			\widehat{Q_3^S \widetilde{w_3}}(\mu_3)
		\, d\mu_1 \, d\mu_2 \, d\mu_3
 	$$
	With a good choice of $M$ based on the Fourier localizations of $u_1$, $v_2$, and $w_3$, we can show that the term with all low modulation factors vanishes, and then use \eqref{Q.big1} on each of the remaining seven terms to obtain estimates that exploit the high modulation.
	
	We begin with the resonance identity (\eqref{resonance} with some changes of variable). For $i=1,2,3$ let
	$$\lambda_i = \tau_i + \beta \xi_i^5 - \alpha \xi_i^3 + \xi_i^{-1}\eta_i^2$$
	where $\tau_i, \xi_i, \eta_i$ are Fourier variables corresponding to $(t,x,y)$. If $\mu_1+\mu_2+\mu_3=0$ and $\beta=-1$ we then have
	\begin{equation}
		\label{KPII5.resonance}
		\lambda_1 + \lambda_2 + \lambda_3 = 
				- 3\alpha \xi_1\xi_2 \xi_3 + 
					5\beta(\xi_1\xi_2\xi_3)(\xi_1^2 + \xi_1 \xi_2 + \xi_2^2) - 
					\frac{(\xi_2\eta_1 - \eta_2 \xi_1)^2}{\xi_1 \xi_2 \xi_3}
	\end{equation}	
	where each of the three right-hand terms has the same sign provided that $\beta <0$ and $\alpha>0$. If $u_1,v_2,w_3$ have Fourier localizations with $|\xi_i| \geq N_i/2$ for $i=1,2,3$, then we have
	\begin{equation}
		\label{lambda.M}
		 |\lambda_1 + \lambda_2 + \lambda_3| \geq 3N_1N_2N_3/8
	\end{equation}	
	(notice that this lower bound makes no use of the summand in \eqref{KPII5.resonance} coming from the fifth-order differential operator). Hence, choosing 
	\begin{equation}
	\label{M.choice1}
		M = N_1 N_2 N_3/8
	\end{equation}
 we conclude that the low modulation term
	$$ \int_\R \langle Q^S_{<M}\widetilde{u_1} Q^S_{<M}\widetilde{v_2}, Q_{<M}^S \widetilde{w_3} \rangle \, dt = 0
	$$
	leaving only integrals with at least one high-modulation factor to be estimated.
	
	We now sketch the proofs of estimates \eqref{pre-bilinear-1} and \eqref{pre-bilinear-2}, referring the reader to the proof of Proposition 3.1 in \cite{HHK09} for further details. We need to estimate expressions of the form
	$$ \left| 
			\sum_{N_1 \lesssim N_2} 
				\int_\R \langle 
							Q_1^S \widetilde{u_{N_1}} 
							Q_2^S \widetilde{v_{N_2}}, 
							Q_3^S \widetilde{w_{N_3}} 
						\rangle 
					\, dt \right| $$
	where at least one $Q_i^S$ is a high-modulation projection $Q^S_{\geq M}$. In the case $i=1$ we may estimate
	\begin{align}
	\label{J1.Q1}
		\Biggl| 
			\sum_{N_1 \lesssim N_2} &
				\int_R \langle Q^S_{\geq M} \widetilde{u_{N_1}}
								Q_2^S \widetilde{v_{N_2}}
								Q_3^S \widetilde{w_{N_3}}
						\rangle \, dt
		\Biggr|	\\
			&\lesssim 
			\norm[L^2]{\sum_{N_1 \lesssim N_2} Q^S_{\geq M} \widetilde{u_{N_1}}}
			\norm[L^4]{Q_2^S \widetilde{v_{N_2}}}
			\norm[L^4]{Q_3^S \widetilde{w_{N_3}}}
			\nonumber\\
			&\lesssim 
			\left( \sum_{N_1 \lesssim N_2}(N_1N_2N_3)^{-1} \norm[V^2_S]{\widetilde{u_{N_1}}}^2\right)^\frac12
			\norm[V^2_S]{\widetilde{v_{N_2}}}
			\norm[V^2_S]{\widetilde{w_{N_3}}}
			\nonumber \\
			&\lesssim
			\left(\sum_{N_1 \lesssim N_2} N_1^{-1} \norm[V^2_S]{u_{N_1}}^2\right)^\frac12
			N_2^{-\frac12} \norm[V^2_S]{v_{N_2}} N_3^{-\frac12} \norm[V^2_S]{w_{N_3}}
			\nonumber
	\end{align}
	where we used \eqref{Q.big1} and \eqref{VP.S1} in the third step.
	
	In the case $i=2$ we can use the Strichartz estimate \eqref{Vp.Bi} on the first and third factors, together with \eqref{Q.big1} on the second factor, to obtain
	\begin{align}
		\label{J1.Q2}
		\Biggl| 
				\int_R &\langle Q_1^S \widetilde{u_{N_1}}
								Q_{\geq M}^S \widetilde{v_{N_2}}
								Q_3^S \widetilde{w_{N_3}}
						\rangle \, dt
		\Biggr|	\\
			&\lesssim \norm[L^2]{Q_{\geq M}^S \widetilde{v_{N_2}}}
						\left( \frac{N_1}{N_3} \right)^\frac14 
							\norm[V^2_S]{\widetilde{v_{N_1}}}
							\norm[V^2_S]{\widetilde{w_{N_3}}}
			\nonumber \\
			&\lesssim 
			(N_1 N_2 N_3)^{-\frac12} \norm[V^2_S]{\widetilde{v_{N_2}}}
				\left( \frac{N_1}{N_3} \right)^\frac14 
				\norm[V^2_S]{\widetilde{u_{N_1}}}
							\norm[V^2_S]{\widetilde{w_{N_3}}}
			\nonumber 
	\end{align}
	which can then be summed over $N_1 \lesssim N_2$ using the Schwartz inequality and the fact that $N_2 \sim N_3$. 
	
	In the case $i=3$, we use the bilinear Strichartz inequality on the $u_{N_1}$ and $v_{N_2}$ terms and use \eqref{Q.big1} on the $w_{N_3}$ factor.
	
	The estimates \eqref{J1.Q1}, \eqref{J1.Q2}, and the analogous $i=3$ estimate together prove \eqref{pre-bilinear-1}. 
	
	The proof of \eqref{pre-bilinear-2} is analogous.

	\end{proof}

\subsection{$L^2$ Estimate}

Next, we prove a stronger estimate on the form $I_T(u,v)$ which will help us prove persistence of regularity in $L^2$. Here is the first time that we make use of the stronger smoothing properties of the fifth-order linear term.

\begin{proposition}
	\label{prop:IT.L2}
	There is a $C>0$ independent of $T$ so that, for all $u,v \in \dot{Y}^{-\frac12} \cap C(\R,H^{1,1}(\R^2))$, the estimate
	\begin{equation}
		\label{IT.L2}
		\norm[\dot{Z}^0]{I_T(u,v)} \lesssim \norm[\dot{Y}^{-\frac12}]{u} \norm[\dot{Y}^{-\frac12}]{v}
	\end{equation}
	holds.
\end{proposition}

\begin{proof}
	The proof follows the pattern of the proof of Proposition \ref{prop:IT.L2} but uses the sharper estimates contained in Lemma \ref{lemma:J1J2.L2} below. We only sketch the main ideas. As before, by bilinearity and symmetry it suffices to show that
	\begin{align}
		\label{IT.J1.L2}
		\norm[\dot{Z}^0]{\sum_{N_2} \sum_{N_1 \lesssim N_2} I_T(u,v)} & \lesssim \norm[\dot{Y}^{-\frac12}]{u} \norm[\dot{Y}^{-\frac12}]{v}
		\intertext{and}
		\label{IT.J2.L2}
		\norm[\dot{Z}^0]{\sum_{N_2} \sum_{N_1 \sim N_2} I_T(u,v)} & \lesssim \norm[\dot{Y}^{-\frac12}]{u} \norm[\dot{Y}^{-\frac12}]{v}
	\end{align}
	with implied constants independent of time. We repeat the proof of Proposition \ref{prop:IT.est}, using the estimate \eqref{pre-bilinear-L2-1} to bound \eqref{IT.J1.L2} and using \eqref{pre-bilinear-L2-2} to bound \eqref{IT.J2.L2}.
\end{proof}

The following estimates, analogous to those of Lemma \ref{lemma:J1J2}, play a key role in the proof of Proposition \ref{prop:IT.L2}. Here we make use of the smoothing properties of the fifth-order term in \eqref{KPII-5}.

\begin{lemma}
	\label{lemma:J1J2.L2}
	Let $u_{N_1}, v_{N_2}, w_{N_3} \in V^2_{-,S}$ so that $\supp \, u_{N_1} \in A_{N_1}$, 
	$\supp \, v_{N2} \in A_{N_2}$, and $\supp \, w_{N_3} \in A_{N_3}$ for dyadic integers $N_1$, $N_2$, $N_3$. For any $T>0$:
	\begin{enumerate}
		\item If $N_2 \sim N_3$ then
				\begin{multline}
					\label{pre-bilinear-L2-1}
					\left| \sum_{N_1 \lesssim N_2} \int_0^T \langle u_{N_1} v_{N_2}, \diff_x w_{N_3} \rangle_{L^2} \, dt  \right| \\
						\lesssim 	
							\left( 
								\sum_{N_1 \lesssim N_2} N_1^{-1} \norm[V^2_S]{u_{N_1}}^2 
							\right)^\frac12 
							N_2^{-\frac12} \norm[V^2_S]{v_{N_2} }
							\norm[V^2_S]{w_{N_3}}
				\end{multline}

		\item If $N_1 \sim N_2$ then
				\begin{multline}
					\label{pre-bilinear-L2-2}
					\sum_{N_3 \lesssim N_2}	
						\sup_{\norm[V^2_S]{w} \leq 1}
							\left| \int_0^T \langle u_{N_1} v_{N_2}, \diff_x w_{N_3} \rangle \, dt\right| \\
						\lesssim 
							N_1^{-\frac12} \norm[V^2_S]{u_{N_1}} N_2^{-\frac12} \norm[V^2_S]{v_{N_2}}
				\end{multline}

	\end{enumerate}
\end{lemma}

\begin{proof}
	We repeat the strategy of proof explained in the proof of Lemma \ref{lemma:J1J2} but use the resonance relation \eqref{KPII5.resonance} in a different way, with a different choice of $M$ in the decomposition. Considering integrals of the form
	\begin{equation}
		\label{int.3Q}
			\int_\R \langle (Q_1^S \widetilde{u_{N_1}}) (Q_2^S\widetilde{v_{N_2}}), (Q_3^S\widetilde{w_{N_3}}) \rangle_{L^2} \, dt 
	\end{equation} 
	we exploit the \emph{two} resonance relations
	\begin{align}
		\label{KPII5.resonance.1}
		\lambda_1 + \lambda_2 + \lambda_3 
				& =- 3\alpha \xi_1\xi_2 \xi_3 
					+ 
					5\beta(\xi_1\xi_2\xi_3)(\xi_1^2 + \xi_1 \xi_2 + \xi_2^2)\\
				 &\quad - \frac{(\xi_2\eta_1 -  \eta_2 \xi_1)^2}{\xi_1 \xi_2 \xi_3}
		\nonumber\\
		\label{KPII5.resonance.2}
		\lambda_1 + \lambda_2 + \lambda_3 
			&= 
				- 3\alpha \xi_1\xi_2 \xi_3 + 
					5\beta(\xi_1\xi_2\xi_3)(\xi_2^2 + \xi_2 \xi_3 + \xi_3^2)\\
			&\quad  - 
					\frac{(\xi_2\eta_3 -  \eta_2 \xi_3)^2}{\xi_1 \xi_2 \xi_3}
		\nonumber 
	\end{align}
	(note that \eqref{KPII5.resonance.2} is obtained from \eqref{KPII5.resonance.1} by permutation of indices). Suppose that $\beta=-1$ and $|\xi_i| \geq N_i/2$, $i=1,2,3$. From \eqref{KPII5.resonance.1} we deduce that
	\begin{equation}
		\label{KPII5.Mbd.1}
			\frac{1}{8} N_1N_2N_3 \leq |\xi_1| \, |\xi_2| \, |\xi_3| \leq \max(|\lambda_1|,|\lambda_2|,|\lambda_3|)
	\end{equation}
	while from \eqref{KPII5.resonance.2} with $\beta=-1$ we obtain
	\begin{align}
		\label{KPII5.Mbd.2}
			\frac{5}{3(2^5)}N_1N_2^2N_3^2 
				&	\leq \frac53 |\xi_1| \, |\xi_2|^2 |\xi_3|^2 \\
				& 	\leq \frac53 |\xi_1 \xi_2 \xi_3 (\xi_2^2 + \xi_2 \xi_3 + \xi_3^2)| 
				\nonumber \\
				&	\leq  \max( |\lambda_1|,|\lambda_2|,|\lambda_3| )
				\nonumber 
	\end{align}
		
	Let $M_1 = N_1N_2N_3/8$ and $M_2 = \frac{5}{3(2^5)}N_1 N_2^2 N_3^2$. If we now take 
	$$M = M_1^\frac12 M_2^\frac12 = \frac{1}{16}\sqrt{\frac53}N_1 N_2^\frac32 N_3^\frac32$$
	we conclude that
	\begin{equation}
		\label{KPII5.Mbd.3}
		M \leq \max(|\lambda_1|, |\lambda_2|, |\lambda_3|)
	\end{equation}
	so that, with this choice of $M$, the integral \eqref{int.3Q} is zero if $Q_i^S = Q_{<M}^S$ for $i=1,2,3$. This means that, to prove the estimates \eqref{pre-bilinear-L2-1} and \eqref{pre-bilinear-L2-2}, we can use the same strategy as before but with the new choice of $M$. We then repeat the estimates used before with the new choice of $M$, and obtain bounds adapted to bounding $I_T(u,v)$ in the  $\dot{Z}^0$ norm. Complete details are given in Chapter 6, Lemmas 6.0.5 and 6.0.7 of \cite{CS2023}.

\end{proof}

\subsection{Estimates on $I_\infty$}

Finally, we prove boundedness of the bilinear map $I_\infty$ and convergence of $I_T$ to $I_\infty$. The following Proposition and its proof closely follow the statement and proof of \cite[Corollary 3.4]{HHK09}.
 
\begin{proposition}
	\label{prop:Iinfty}
	The map $I_\infty$ extends to a bounded linear map 
	$$ I_\infty: \dot{Y}^{-\frac12} \times \dot{Y}^{-\frac12} \to \dot{Z}^{-\frac12} $$
	with the property that for all $u,v \in \dot{Y}^{-\frac12}$, 
	\begin{equation}
		\lim_{T \to \infty} \norm[\dot{Z}^{-\frac12}]{I_T(u,v) - I_\infty(u,v)} = 0.
	\end{equation}
	Finally, for any $u \in \dot{Y}^{-\frac12}$, $\varphi^+ = \lim_{t \to \infty} S(-t) I_\infty(u,u)$ exists in $\dot{H}^{-\frac12,0}(\R^2)$.
\end{proposition}

\begin{remark}
	The existence of $\varphi_+$ implies existence of scattering asymptotes if $u$ is a solution of \eqref{KPII-5-int}. A similar result holds for the limit as $t \to -\infty$. 
\end{remark}

For the proof we will use Lemma \ref{lemma.Y.scat} and the following two lemmas. 
These will be used in the proof that $I_T(u,v) \to I_\infty(u,v)$ in $\dot{Z}^{-\frac12}$ for $u,v \in \dot{Y}^{-\frac12}$. Here and in what follows, we introduce the cutoff function
	\begin{equation}
		\label{alpha.T}
		\alpha_T(t) = 
			\begin{cases}
				0,		&	t < T-1 \\
				t-T+1, 	& T-1 \leq t \leq T,\\
				1		&	t \geq T.
			\end{cases}
	\end{equation}

\begin{lemma}[\cite{CS2023}, Lemma 5.2.3]
	\label{lemma:free.shrink}
	Suppose that $\varphi \in L^2(\R)$, fix $M > 0$, and let $\eps>0$ be given. There exists $T'>0$ depending on $\varphi$ so that for $Q^S_* = Q^S_{< M}$ or $Q^S_* = Q^S_{\geq M}$, 
	\begin{equation}
		\label{free.shrink}
		\norm[L^4]{Q^S_* \alpha_T S(t) \varphi} \lesssim \eps \norm[L^2]{\varphi}
	\end{equation}
	for all $T > T'$.
\end{lemma}

\begin{proof}
It is enough to show that $\norm[L^4]{Q_*^S\alpha_T S(t)\varphi} < \eps$ for given $\eps$ and $T$ sufficiently large. First, as $S(t) \varphi \in L^4(\R^3)$ by \eqref{S1.complete}, it follows by dominated convergence that $\norm[L^4]{\alpha_T S(t) \varphi} \to 0$ as $T \to \infty$, so it remains to show that this remains true when premultiplied by $Q^S_{< M}$. The statement for $Q^S_{\geq M} = I - Q^S_{<M}$ is an immediate consequence. 

By definition, $Q^S_{<M} \alpha_T S(t) \varphi = S(t) (Q_{< M} \alpha_T) \varphi$. Recall that $Q_{<M}$ is a Fourier multiplier with symbol $\varphi(\tau/M)$ 	where $\varphi \in C_0^\infty([-2,2])$ and $\varphi(\tau)=1$ for $|\tau| \leq 1$.  Hence
$$ (Q_{<M} \alpha_T)(t) = \int \widecheck{\varphi}(t') \alpha_T(t - t'/M) \, dt' $$
is a bounded function of $t$ which commutes with $S(t)$ and does not affect the dominated convergence argument. This gives the desired estimate.
\end{proof}

We will also need an improved bilinear estimate analogous to \eqref{BS-1}. As before $Q^S_i$ denotes one of $Q^S_{<M}$ or $Q^S_{\geq M}$.

\begin{lemma}[\cite{CS2023}, Lemma 5.2.4]
	\label{lemma:BS-1-bis}
	Suppose that $\varphi \in \dot{H}^{-\frac12,0}(\R^2)$, $v \in V^2_S$, and $u$ is such that $P_{N_1}u = P_{N_1} \alpha_TS(t) \varphi$. If we choose $T$ so that 
	\begin{equation}
		\label{BS-1-bis-est}
		\norm[L^4(\R^3)]{ Q_1^S P_{N_1} u} \lesssim \eps \norm[L^2(\R^2)]{P_{N_1} \varphi},
	\end{equation}
	then for $N_1 \lesssim N_2$, the estimate
	\begin{multline}
		\label{BS-1-bis}
		\norm[L^2(\R^3)]{(Q_1^S P_{N_1}u) (Q^S_2 P_{N_2}v)}\\
		\lesssim \eps^\frac12 \left( \frac{N_1}{N_2}\right)^{\frac14}	
			\left( 1+ \log\left(\frac{N_2}{N_1}\right) \right)
			\norm[L^2(\R^2)]{P_{N_1}\varphi}
			\norm[V^2_S]{P_{N_2} v}.
	\end{multline}

\end{lemma}

\begin{proof}
	The proof is an interpolation between the two estimates 
	\begin{align}
		\label{BS-1-bis-est1}
		\norm[L^2(\R^3)]{(Q_1^S P_{N_1}u) (Q^S_2 P_{N_2}v)}
			& \leq \eps \norm[L^2(\R^2)]{P_{N_1}\varphi}	
						\norm[V^2_S]{P_{N_2} v}
	\intertext{and}
		\label{BS-1-bis-est2}
		\norm[L^2(\R^3)]{(Q_1^S P_{N_1}u) (Q^S_2 P_{N_2}v)}
			&\leq \left( \frac{N_1}{N_2}\right)
					\left( 1+ \log \left( \frac{N_2}{N_1} \right) \right)^2
					\norm[L^2]{P_{N_1} \varphi} 
					\norm[V^2_S]{P_{N_2} v}.
	\end{align}
	The first follows from H\"older's inequality, the linear Strichartz estimate, \eqref{BS-1-bis-est}, and \eqref{L4.Vp}. The second follows from the bilinear Strichartz estimate \eqref{Vp.Bi}.

\end{proof}

\begin{proof}[Proof of Proposition \ref{prop:Iinfty}]
	First, if $u,v \in \dot{Y}^{-\frac12}$, then $P_N I_\infty(u,v) \in V^2_S$ by Proposition \ref{prop:V0}, Proposition \ref{prop:IT.est}, and the definition of $V^2_S$ spaces. Using Proposition \ref{prop:V0} again we recover
	$$ \sum_N N^{-1} \norm[V^2_S]{P_N I_T(u,v)}^2 < \infty$$
	which shows that $I_\infty(u,v) \in \dot{Y}^{-\frac12}$. In the last step of the proof we will show that, also, $I_\infty(u,v) \in \dot{Z}^{-\frac12}$. 
	
	Next, we show the existence of $\varphi_+$. If $u \in \dot{Y}^{-\frac12}$, we have $I_\infty(u,u) \in \dot{Y}^{-\frac12}$. It follow from Lemma \ref{lemma.Y.scat} that there is a $\varphi^+ \in \dot{H}^{-\frac12,0}$ with the required properties.
	
	Finally, we show that $I_T \to I_\infty$ and show that $I_\infty(u,v) \in \dot{Z}^{-\frac12}$. Before giving the proof we make some simplifying remarks and outline the main ideas. First, by continuity and density, it is enough to prove the convergence for those $u,v$ for which only finitely many of the $P_N u$ and $P_N v$ are nonzero. Second, it is enough to show that the sequence $\{ I_T(u,v) \}_T$ is Cauchy in $\dot{Z}^{-\frac12}$ since $\dot{Z}^{-\frac12}$ is complete. Third, by Lemma \ref{lemma.Y.scat}, there are $u_+$ and $v_+$ in $\dot{H}^{-\frac12,0}$  so that $S(-t) u(t) \to u_+$ and $S(-t)v(t) \to v_{+}$ in $\dot{H}^{-\frac12,0}$ as $t \to \infty$, and we can reduce the proof of the Cauchy estimate to a study of $I_T$ on solutions of the linear problem. The proof of convergence will also show that $I_\infty(u,v) \in \dot{Z}^{-\frac12}$.
	
	Indeed, suppose that $u = \sum_{L^{-1} \leq N \leq L} u_N$ and $v= \sum_{L^{-1} \leq N \leq L} v_N$ for a dyadic integer $L$. By Lemma \ref{lemma.Y.scat}, there exist $u_+,v_+ \in \dot{H}^{-\frac12,0}$ so that 
	$$\lim_{t \to +\infty} \norm[\dot{H}^{-\frac12,0}]{u(t) - S(t) u_+}=0$$ 
	and
	$$\lim_{t \to +\infty} \norm[\dot{H}^{-\frac12,0}]{v(t) -  S(t)v_+}=0$$
	We set
	$$ \psi(t) =  S(t) u_+, \quad \phi(t) =  S(t) v_+ .$$
		
	For $T>0$, let
	$$ \widetilde{u}(t) = u(t) - \alpha_0(t) \phi(t), \quad \widetilde{v}(t) = v(t) - \alpha_0(t) \psi(t). $$
	Note that $\lim_{t \to \infty} \norm[\dot{H}^{-\frac12,0}]{\widetilde{u}(t)} = \lim_{t \to \infty} \norm[\dot{H}^{-\frac12,0}]{\widetilde{v}(t)} = 0$. It follows from Proposition 2.4(i) in \cite{HHK09} with $0$ replaced by $\infty$ that, for given $\eps>0$, there is a $T>0$ with 
		\begin{equation}
		\label{uv.tilde.vanish}
			\norm[\dot{Y}^{-\frac12}]{\alpha_T \widetilde{u}}, \norm[\dot{Y}^{-\frac12}]{\alpha_T \widetilde{v}} < \eps.
		\end{equation}
	(recall $\alpha_T$ was defined in \eqref{alpha.T}). A straightforward computation using the definition of $I_T(u,v)$ shows that, for $0 < T_1 < T_2 < T$,
	\begin{equation}
		\label{IT1.IT2}
		I_{T_1}(u,v) - I_{T_2}(u,v) = I_{T_1}(\alpha_{T_1}u,\alpha_{T_1}v) - I_{T_2}(\alpha_{T_1}u,\alpha_{T_1}v).
	\end{equation}
	We will show that $\norm[\dot{Z}^{-\frac12}]{I_{T_1}(u,v) - I_{T_2}(u,v)} \to 0$ as $T_1,T_2 \to \infty$ by showing that each one of the two terms is small for $T_1,T_2$ sufficiently large. We give the estimate for $I_{T_1}$ since the estimate for $I_{T_2}$ is similar. We write
	\begin{align*}
		I_{T_1}(\alpha_{T_1}u,\alpha_{T_1}v) 
			&=	I_{T_1}(\alpha_{T_1} \widetilde{u},v) + I_{T_1}(\alpha_{T_1} \psi(t),\widetilde{v}) + I_{T_1}(\alpha_{T_1}(\psi(t),\alpha_{T_1}\phi(t))
	\end{align*}
	The first two terms vanish in $\dot{Z}^{-\frac12}$ norm as $T_1 \to \infty$ by \eqref{IT.bdYY} and \eqref{uv.tilde.vanish}. Thus, it remains to show that 
	$I_{T_1}(\alpha_{T_1}\psi(t),\alpha_{T_1}\phi(t))$ vanishes as $T_1 \to \infty$. To do this we will need estimates on the bilinear form that take into account properties of the linear evolution. 
	
	Since $\psi(t)$ and $\phi(t)$ are finite linear combinations of solutions to the linear problem \eqref{KPII-5-linear}, we may use Lemma \ref{lemma:free.shrink} to conclude that there is a $T'$ so that 
	\begin{align}
		\label{free.shrink.psi}
		\norm[L^4]{Q^S_*\alpha_{T_1}\psi(t)} &\lesssim \eps \norm[\dot{H}^{-\frac12,0}]{\psi(0)}
		\intertext{and}
		\label{free.shrink.phi}
		\norm[L^4]{Q^S_* \alpha_{T_1}\phi(t)} &\lesssim \eps \norm[\dot{H}^{-\frac12,0}]{\phi(0)}
	\end{align}
	for $T_1 > T'$.
	
	We now claim that 
	$$ \norm[\dot{Z}^{-\frac12}]{I_T(\alpha_{T_1}\psi(t),\alpha_{T_1}\phi(t))} \lesssim \eps^\frac12 \norm[\dot{H}^{-\frac12,0}]{\psi(0)} \norm[\dot{H}^{-\frac12,0}]{\phi(0)} $$
	for $T_1 > T'$. To see this, we rework the estimates on
	\begin{align*}
		J_1	&=	\sum_{N_2} \sum_{N_1 \ll N_2} I_T(\alpha_{T_1}\psi(t),\alpha_{T_1}\phi(t))
	\intertext{and}
		J_2 &= 	\sum_{N_2} \sum_{N_1 \sim N_2} I_T(\alpha_{T_1}\psi(t),\alpha_{T_1}\phi(t)
	\end{align*}
	(see the proofs of Propositions \ref{prop:IT.est} and \ref{prop:IT.L2}) using the $L^4$ estimates \eqref{free.shrink.psi} and \eqref{free.shrink.phi} at key points to obtain the needed factors of $\eps$.  We will outline the estimates on $J_1$ and $J_2$ but refer the reader to the proof of Theorem 5.3.8 in \cite{CS2023} for details.
	
	To bound $J_1$ we repeat the proof of \eqref{IT.J1.Final} replacing \eqref{pre-bilinear-1} with the improved estimate
	\begin{multline*}
		\Biggl| \sum_{N_1 \lesssim N_2} \int_0^T \int_{\R^2} \alpha_{T_1}(P_{N_1}\psi) \, \alpha_{T_1} (P_{N_2} \phi) \,  \diff_x w_{N_3} \, dx \, dy \, dt
		\Biggr|\\
		\lesssim \eps^\frac12 \norm[\dot{H}^{-\frac12,0}]{\psi(0)}
			N_2^{-\frac12}\norm[L^2]{P_{N_2}\phi(0)} N_3^{\frac12}\norm[V^2_S]{w_{N_3}}
	\end{multline*}
	using \eqref{free.shrink.psi},  \eqref{free.shrink.phi}, and \eqref{BS-1-bis}. This gives
	\begin{equation}
		\label{IT.J1.asy.shrink}
		\norm[\dot{Z}^{-\frac12}]{J_1} \lesssim \eps^{\frac12} \norm[\dot{H}^{-\frac12,0}]{\psi(0)} \norm[\dot{H}^{-\frac12,0}]{\phi(0)}.
	\end{equation}
	
	To bound $J_2$ we repeat the proof of \eqref{IT.J2.Final} replacing \eqref{pre-bilinear-2} with the improved estimate
	\begin{multline*}
		\Biggl(
			\sum_{N_3 \lesssim N_2}
				N \sup_{\norm[V^2_S]{w}=1}	
			\left|
				\int_0^T \int 
					\alpha_{T_1} (P_{N_1}\psi) \,
					\alpha_{T_2} (P_{N_2}\phi \,
					 w_{N_3}
				\, dx \, dy \, dty
			\right|	
		\Biggr)^\frac12\\
		\lesssim 
			\eps^{\frac12} 
				N_1^{-\frac12}  \norm[L^2]{P_{N_1}\psi(0)}
				N_2^{-\frac12}  \norm[L^2]{P_{N_2}\phi(0)}
	\end{multline*}
	which again uses \eqref{free.shrink.psi},  \eqref{free.shrink.phi}, and \eqref{BS-1-bis}. 
	
	\end{proof}

\section{Proof of the Main Theorem}

In this section we give the proof of Theorem \ref{thm:main}. For a discussion of calculus on maps between Banach spaces, see, for example,  Appendices A and B of P\"{o}schel-Trubowitz \cite{PT1987}.

\begin{proof}[Proof of Theorem \ref{thm:main}]
	We will construct the solution for $t \in [0,\infty)$ and prove scattering as $t \to +\infty$. The analogous proof for $t \in (-\infty,0]$ and scattering as $t \to -\infty$ is similar and is omitted.

	Denote by $\dot{Z}^{-\frac12}[0,\infty)$ the set of functions $u$ with $u = \left. v \right|_{[0,\infty)}$ for some $v \in \dot{Z}^{-\frac12}$, and define
	$$ \norm[\dot{Z}^{-\frac12}[0,\infty)]{u} = 
			\inf 
				\left\{ 
					\norm[\dot{Z}^{-\frac12}]{v} : \left. v \right|_{[0,\infty)} = u
				\right\} .
	$$ 
	We can make an analogous definition of $\dot{Y}^{-\frac12}[0,\infty)$. We have the  continuous embedding $\dot{Z}^{-\frac12}[0,\infty) \hookrightarrow \dot{Y}^{-\frac12}[0,\infty)$.
	
	For $r>0$, let 
	\begin{equation}
		\label{Dr}
			D_r = \{ u \in \dot{Z}^{-\frac12}[0,\infty): \norm[\dot{Z}^{-\frac12}[0,\infty)]{u} < r \}. 
	\end{equation}
	and, for $\delta > 0$, define $\dot{B}_\delta$ as in \eqref{dotB}. Observe that, if $u_0 \in \dot{B}_\delta$, the function $\alpha_0(t) S(t) u_0$ lies in $\dot{Z}^{-\frac12}$ so that the restriction of $S(t) u_0$ to $[0,\infty)$ lies in $\dot{Z}^{-\frac12}[0,\infty)$.  With this observation in mind, we can prove existence and uniqueness of solutions in $\dot{Z}^{-\frac12}[0,\infty)$ by a contraction mapping argument as follows.
	
	For $\delta>0$ and $r$ to be chosen let
	\begin{align}
		\label{contraction}
		\Phi:	(u,u_0) &\to S(t)u_0 - I_T(u,u) \\
			D_r \times \dot{B}_\delta &\mapsto \dot{Z}^{-\frac12}[0,\infty)
			\nonumber 
	\end{align}
	where $\dot{B}_\delta$ and $D_r$ are defined as in \eqref{dotB} and \eqref{Dr}. We need to show that, for each fixed $u_0 \in \dot{B}_\delta$, 
	\begin{align}
		\label{contract1}
			\Phi: D_r &\to D_r \\
		\intertext{and}
			\norm[\dot{Z}^{-\frac12}]{\Phi(u,u_0) - \Phi(v,u_0)} &\leq c \norm[\dot{Z}^{-\frac12}]{u-v}, \quad 0<c<1.
	\end{align} 
	for all $u,v \in D_r$. This argument proceeds exactly as in \cite{HHK09} with the choices $\delta=(4C+4)^{-2}$ and $r=(4C+4)^{-1}$, where $C$ is the constant in \eqref{IT.bdZZ}. We conclude that $\Phi$ has a unique fixed point, producing a unique solution to \eqref{KPII-5-int}.
	
	To show continuous dependence on $u_0$, we use the integral equation \eqref{KPII-5-int} together with the estimate \eqref{IT.bdZZ} to conclude that, if $u_1$ and $u_2$ are solutions with respective initial data $u_1(0), u_2(0) \in \dot{B}_\delta$, then
	\begin{align}
		\label{u.L2.est2}
		\norm[\dot{Z}^{-\frac12}]{u_1-u_2} 
		&\leq \norm[\dot{H}^{-\frac12,0}]{u_1(0)-u_2(0)} \\
		&\qquad + C \norm[\dot{Z}^{-\frac12}]{u_1-u_2} \left( \norm[\dot{Z}^{-\frac12}]{u_1} + \norm[\dot{Z}^{-\frac12}]{u_2} \right).
		\nonumber
	\end{align}
	With the above choices of $r$ and $\delta$ we may conclude that
	\begin{equation}
		\label{continuity.1}
			\norm[\dot{Z}^{-\frac12}[0,\infty)]{u_1 - u_2} \lesssim \norm[\dot{H}^{-\frac12,0}]{u_1(0)-u_2(0)}
	\end{equation}

	As $\dot{Z}^{-\frac12}[0,\infty)$ in continuously embedded in $\dot{Y}^{-\frac12}[0,\infty)$, 
	any solution $u$ of \eqref{KPII-5-int} satisfies 
	$$ S(-t) u = u_0 - S(-t) I_{\infty}(u,u) $$
	where $I_\infty(u,u) \in \dot{Y}^{-\frac12}[0,\infty)$. It follows from Lemma \ref{lemma.Y.scat} that the second term has an asymptotic limit as $t \to \infty$ in $\dot{H}^{-\frac12,0}$ so we define
	\begin{equation}
		\label{u+}
		u_+ = u_0 + \lim_{t \to \infty} S(-t)I_\infty(u,u)
	\end{equation}
	where the limit is taken in $\dot{H}^{-\frac12,0}(\R^2)$. One can give an analogous definition for $u_-$, the asymptote as $t \to -\infty$. 
	The map $u_0 \mapsto u_+$ is the composition of $u_0 \mapsto u$ and the map \eqref{u+}. The first is analytic by the analytic implicit function theorem since $u$ is defined implicitly by
	$\Phi(u_0,u)=u$ and $\Phi$ is an analytic map from $D_r$ to itself for each fixed $u_0 \in \dot{B}_\delta$. The second is easily seen to be differentiable in the complex sense, hence analytic. Hence $W_+:u_0 \to u_+$ is analytic. The map $V_+$ is 
	analytic by the analytic inverse function theorem since it is easy to compute that $V_+(0)=0$ and $DV_+(0)=I$, the identity map. 
	
	A similar argument proves the existence and uniqueness of solutions in $\dot{Z}^{-\frac12}(-\infty,0]$ for $u_0 \in \dot{B}_\delta$, existence and analyticity of the map $W_-: u_0 \to    u_-$, and analyticity of the map $V_-: u_- \to u_0$.   We have thus established existence and uniqueness of solutions $u \in \dot{Z}^{-\frac12}$ for $u_0 \in \dot{B}_\delta$ together with scattering properties. Using the estimate \eqref{continuity.1} and a similar estimate for the solutions on $(-\infty,0]$, we conclude that
	\begin{equation}
		\label{continuity.2}
			\norm[C(\R,\dot{H}^{-\frac12,0})]{u_1 - u_2}
				\lesssim \norm[\dot{Z}^{-\frac12}]{u_1-u_2} 
				\lesssim \norm[\dot{H}^{-\frac12,0}]{u_1(0)-u_2(0)}
	\end{equation}
	which gives continuity of the solution in terms of the initial data.
	
	It remains to show that, if $u_0 \in \dot{B}_\delta \cap L^2(\R^2)$, then 
	$\norm[L^2]{u(t)} = \norm[L^2]{u(0)}$ and $\norm[L^2]{u_+} = \norm[L^2]{u_0}$. First, we show that $u \in C(\R,L^2(\R))$. From the integral equation we may estimate
	$$
		\norm[\dot{Z}^0]{u} \leq \norm[L^2]{u_0} + \norm[\dot{Z}^0]{I_\infty(u,u)}
	$$
	where the second right-hand term is finite owing to Proposition \ref{prop:IT.L2}. Since $u \in \dot{Z}^0 \subset \dot{Y}^0$, it follows that $u(t) \in L^2(\R^2)$ for all $t$. To recover continuity of $u:\R \to L^2(\R^2)$ we exploit the fact that $u \in C(\R,\dot{H}^{-\frac12,0})$ and the uniform bound on $\sum_N \norm[L^2(\R^2)]{P_N u(t)}^2$. Thus $u \in C(\R,L^2(\R^2))$. 
	
	We next show that $\norm[L^2]{u(t)}$ is conserved. For given $u(0) \in \dot{H}^{-\frac12,0}(\R^2) \cap L^2(\R^2)$, we choose a sequence $\{ u_n(0) \}$ from $C_0^\infty(\R^2)$ with $u_n(0) \to u(0)$ in $\dot{H}^{-\frac12,0} \cap L^2$. We claim that, first, the solutions $u_n(t)$ with initial data $u_n(0)$ have conserved $L^2$ norm and second, $u_n(t) \to u(t)$ in $L^2 \cap \dot{H}^{-\frac12,0}$ for each $t$. If correct, the $L^2$ conservation law for $u(t)$ follows.
	
	For $u_n(0) \in C_0^\infty$ we obtain a classical solution for which the conservation of $L^2$ norm follows from multiplying \eqref{KPII-5} by $u$ and integrating over space. It remains to show the claimed continuity.
	
	Let $u_1(t)$ and $u_2(t)$ be solutions of \eqref{KPII-5} with initial data $u_1(0)$ and $u_2(0)$ belonging to $\dot{H}^{-\frac12,0} \cap L^2$. From the integral equation \eqref{KPII-5-int}, the estimates \eqref{IT.bdZZ} and \eqref{IT.L2}, and bilinearity of $I_T$  we conclude that
	\begin{align}
	\label{u.L2.est1}
		\norm[\dot{Z}^0]{u_1-u_2} 
		&\leq \norm[L^2]{u_1(0)-u_2(0)} \\
		&\qquad + C \norm[\dot{Z}^{-\frac12}]{u_1-u_2} \left( \norm[\dot{Z}^{-\frac12}]{u_1} + \norm[\dot{Z}^{-\frac12}]{u_2} \right)
		\nonumber
	\end{align}
	With the choices $r=(4C+4)^{-1}$ we deduce from \eqref{u.L2.est1} and \eqref{u.L2.est2} that, for $u_1,u_2 \in \dot{B}_\delta \cap L^2$, the estimate
		$$ \norm[\dot{Z}^0]{u_1 - u_2} \lesssim \norm[L^2]{u_1(0)-u_2(0)} + \norm[\dot{H}^{-\frac12,0}]{u_1(0) - u_2(0)}$$
	holds, which implies the required $L^2$ continuity of solutions. 
\end{proof}

\section*{Acknowledgements}
The authors thank the referees for a number of useful comments and suggestions which have improved our paper.
This work was supported by a grant from the Simons Foundation.

\bigskip

\bibliographystyle{plain}
\bibliography{KPII-5.bib}

\end{document}